\documentclass[preprint,12pt]{elsarticle}

\usepackage{amssymb}
\usepackage{amsmath}
\usepackage{amsthm}
\usepackage{array}
\usepackage{rotating}
\usepackage{vmargin}
\setmarginsrb{2.2cm}{1cm}{2.2cm}{3cm}{1cm}{1cm}{1.5cm}{1.5cm}
\newtheorem{thm}{Theorem}
\newtheorem{cor}{Corollary}
\newtheorem{prop}{Proposition}
\newtheorem{lem}[thm]{Lemma}
\newdefinition{rem}{Remark}
\newdefinition{defi}{Definition}
\newproof{pf}{Proof}
\newcolumntype{M}[1]{>{\raggedright}m{#1}}

\journal{*****}

\begin{document}

\begin{frontmatter}

\title{Integrability and non integrability of some $n$ body problems}
\author{Thierry COMBOT\fnref{label2}}
\ead{thierry.combot@u-bourgogne.fr}
\address{IMB, Universi\'e de Bourgogne, 9 avenue Alain Savary, 21078 Dijon Cedex }

\title{Integrability and non integrability of some $n$ body problems}

\author{}

\address{}

\begin{abstract}
We prove the non integrability of the colinear $3$ and $4$ body problem, for any masses positive masses. To deal with resistant cases, we present strong integrability criterions for $3$ dimensional homogeneous potentials of degree $-1$, and prove that such cases cannot appear in the $4$ body problem. Following the same strategy, we present a simple proof of non integrability for the planar $n$ body problem. Eventually, we present some integrable cases of the $n$ body problem restricted to some invariant vector spaces.
\end{abstract}


\end{frontmatter}

\section{Introduction}

In this article, we will consider the $n$ body problem whose Hamiltonian is given by
$$H_{n,d}=T_{n,d}(p)+V_{n,d}(q)=\sum\limits_{i=1}^n \frac{\lVert p_i \lVert^2}{2m_i} + \sum\limits_{1\leq i <j \leq n} \frac{m_im_j} {\lVert q_i-q_j \lVert}$$
The quadratic form $T$ correspond to kinetic energy, $V$ is the potential, which is a homogeneous function of degree $-1$ in $q$. The coordinates $q_1,\dots,q_n$ correspond respectively to the coordinates of the bodies $m_1,\dots,m_n$.

Already since Poincare and Bruns \cite{46,46b}, it is known that the $n$-body problem is for $n\geq 3$ not integrable in general. Bruns in \cite{46} proved the non-existence of additional algebraic first integrals, later generalized by Julliard-Tosel \cite{47}, and more recent work like \cite{5,68,69} prove the meromorphic non-integrability or non existence of meromorphic first integrals in some cases. All these proofs strongly suggest that the $n$-body problem is never integrable for $n\geq 3$, even in particular cases (as proven for example for the isosceles $3$-body problem in \cite{115}).  The colinear problem (in dimension $1$) is a priori more difficult than the non-integrability proof of the $n$ body problem in the plane and higher dimension, because it needs fewer additional first integrals to be integrable. Remind that as the energy and the impulsion of the center of mass are first integrals, in dimension $1$ we only need $n-2$ additional first integrals for integrability. We will see that even if the problem is not so easy as the planar case, it can be completely studied in the case $n=3,4$ through the bounding of eigenvalues of the Hessian of $V$ at central configurations (see Definition \ref{defcentral}). A similar trick allows to obtain a simple proof of the non integrability of the planar case with positive masses. In the opposite direction, the $n$-body problem also possesses explicit algebraic orbits, linked to central configurations \cite{93}. Restricting the $n$-body problem to a vector space associated to a central configuration leads in particular to an integrable problem, although very simple.
Still, as we will see, there are also not so trivial invariant vector spaces of the $n$-body problem on which the potential is integrable.

In the integrability analysis of the $n$ body problem, and in the more general case of homogeneous potential, the notion of central configuration/Darboux point plays a key role

\begin{defi}\label{defcentral}
We consider the potential $V_{n,d}$ of the $n$ body problem. We will say that $c\in\mathbb{C}^{nd}$ is a central configuration if there exists $g\in\mathbb{C}^d,\alpha \in\mathbb{C}$ such that
$$\frac{\partial}{\partial q_i} V(c_1-g,\dots,c_n-g)= \alpha (c_1-g,\dots,c_n-g) \qquad i=1\dots n$$
The scalar $\alpha$ is called the multiplier. We say that the central configuration is proper if $\alpha\neq 0$ (else called an absolute equilibrium). In the more general setting of $V$ a homogeneous potential of degree $-1$, we call $c$ a Darboux point if moreover $g=0$.
\end{defi}

We add this constant $g$ in our Definition for the $n$ body problem as the potential is in this case invariant by translation, and thus we do not (always) want to require that the center of mass be at $0$. Our non-integrability proofs will be based on variational equations of the corresponding differential system near these central configurations. The main theorem behind these non-integrability proofs is the following

\begin{thm}\label{thmmorales} (Morales, Ramis, Simo \cite{2}) Let $V$ be a meromorphic homogeneous potential of degree $-1$ and $c$ a Darboux point. If $V$ is meromorphically integrable, then the identity component of the Galois group of the variational equation near the homothetic orbit associated to $c$ is abelian at any order. Moreover, the identity component of the Galois group of the first order variational equation is abelian if and only if
$$Sp(\nabla^2 V(c))\subset \left\lbrace\textstyle{\frac{1}{2}}(k-1)(k+2),\; k\in\mathbb{N} \right\rbrace$$
\end{thm}

Remark also that in dimension $1$, $V_{n,1}$ is a rational potential (thus univaluated on $\mathbb{C}^n$), but is not in higher dimension. In the complex domain, the potential $V_{n,d},\; d\geq 2$ is properly defined on an algebraic variety $\mathcal{S}$. An extension of Theorem \ref{thmmorales} has been done in \cite{91}, and proves that in the $n$ body problem, the necessary condition for integrability on the Galois group of variational equations still holds.

Such a Theorem can be either used for each central configuration separately, or simultaneously using some algebraic properties. In the case of the $n$ body problem, a direct computation of central configurations is often too difficult. The colinear case with $n=3,4$ is still tractable, and we prove moreover that a complete computation of central configurations is not necessary, only majorations on eigenvalues of the Hessian matrix of $V_{n,1}$ at Darboux points is necessary.

Using a real algebraic geometry software RAGlib \cite{74}, we prove such a majoration for $n=3,4$ and we conjecture that a similar majoration always hold for any $n$. We then prove very strong non-integrability Theorem that rules out any potential which satisfies these bounds. In the planar case, we also prove a similar majoration, which holds moreover for any $n$. This allows to prove the non-integrability of the planar $n$-body problem. The main theorems of this article are the following

\begin{thm}\label{thmmain1}
For any $(m_1,m_2,m_3)\in{\mathbb{R}_+^*}^3$, the potential $V_{3,1}$ is not meromorphically integrable. Moreover, if $m_1+m_2+m_3=1$, the variational equations near the unique real central configuration have an Abelian Galois group (over the base field $\mathbb{C}(t)$) up to an order
\begin{itemize}
\item greater than $1$ if and only if it exists $\rho\in\mathbb{R}_+^*$ and $k\in\{5,9,14\}$ such that
\begin{small}
\begin{align*}
m_1 & =\frac{(\rho+1)(-8\rho^5+k\rho^5-12\rho^4+3k\rho^4-8\rho^3+3k\rho^3+3k\rho^2+3k\rho+k)}{k(1+2\rho^3+\rho^4+2\rho+\rho^2)^2}\\
m_2 & =-\frac{(-8\rho^4+k\rho^4-28\rho^3+2k\rho^3+k\rho^2-40\rho^2-28\rho+2k\rho-8+k)\rho^2}{k(1+2\rho^3+\rho^4+2\rho+\rho^2)^2} \quad (E_k)\\
m_3 & =\frac{(\rho+1)(k\rho^5+3k\rho^4+3k\rho^3-8\rho^2+3k\rho^2-12\rho+3k\rho-8+k)\rho^2}{k(1+2\rho^3+\rho^4+2\rho+\rho^2)^2}
\end{align*}
\end{small}
\item equal to $2$ if and only if moreover $m_1=m_3$ or $(m_1,m_2,m_3)\in E_9$.
\end{itemize}
\end{thm}

\begin{thm}\label{thmmain2}
For any $m_1,m_2,m_3,m_4>0,\;\;m_1+m_2+m_3+m_4=1$, the potential $V_{4,1}$ is not integrable. Moreover, near the unique real central configuration, there are at most $14$ one dimensional irreductible algebraic curves in the space of masses for which the variational equations have virtually Abelian Galois groups at least up to order $1$. At least one of them, and at most $10$ of them correspond to masses for which the second order variational equations have a virtually Abelian Galois group. None of them have variational equation whose Galois group is virtually Abelian at order $5$.
\end{thm}

\begin{thm}\label{thmmain3}
For any $n$-uplet of positive masses, the planar $n$ body problem is not meromorphically integrable.
\end{thm}

\begin{thm}\label{thmmain4}
The planar $5$ body problem with masses $m=(-1/4,1,1,1,1)$ restricted to the vector space
$$W=\{q\in\mathbb{R}^{10}\!\!,\; q_{1,1}=q_{1,2}=q_{2,1}+q_{4,1}=q_{2,2}+q_{4,2}=q_{3,1}+q_{5,1}=q_{3,2}+q_{5,2}=0\}$$
is integrable in the Liouville sense.

The spatial $n+3$ body problem with masses $m=(m_1,\dots,m_n,-\alpha,4\alpha,4\alpha)$ restricted to the vector space
$$W=\{q\in\mathbb{R}^{3(n+3)},\;\; q_{n+1,1}=q_{n+1,2}=q_{n+1,3}=q_{n+2,1}=q_{n+2,2}=q_{n+3,1}=$$
$$q_{n+3,2}=q_{n+2,3}+q_{n+3,3}=0,\; {q_{i,3}}_{\mid_{i=1\dots n}}=0,\; {q_i}_{\mid_{i=1\dots n}}=\beta R_\theta c,\; \beta,\theta\in\mathbb{R}\}$$
where $c$ is a central configuration of $n$ bodies with masses $(m_1,\dots,m_n)$ in the plane on the unit circle with center of mass at $0$, $R_\theta$ being a rotation in this plane and $\alpha$ chosen such that the configuration $c$ with the central mass $-\alpha$ is an absolute equilibrium is integrable in the Liouville sense.
\end{thm}

The Theorem \ref{thmmain1} implies the non integrability of the colinear $3$ body problem, which was already done in \cite{116} using the systematic approach using all central configurations and a relation between the eigenvalues of Hessian matrices. This approach is hard to apply to more complicated systems as its cost is exponential in the number of central configurations. This is due to the fact that all central configurations are analyzed, even if only a few of them would be probably enough to conclude to non integrability. Also, the physical assumption that the masses are real positive is not used. In the next section, we thus make a more precise analysis of variational equations near the unique real central configuration, whose existence and uniqueness is a result of Moulton \cite{42}:

\begin{thm}\label{thmmoult} (Moulton \cite{42})
For any fixed positive masses $m_1,\dots,m_n$ with a fixed order of the masses, the colinear $n$ body problem admits exactly one real central configuration.
\end{thm}

Remark that also in the not trivially integrable example we found, central configurations seem to play a key role. In particular, they all contain continuums of central configurations (the first case contains the famous $5$ body central configuration of Roberts \cite{94b}). According to a conjecture of Smale, proved for $n=4,5$ in \cite{43,93}, such continuums are not possible with positive masses.

\section{The colinear $3$ body problem}

\subsection{Central configurations}

\begin{prop} \label{thmeuler} (Euler)
We pose $c=(-1,0,\rho)$ with $\rho\in\mathbb{C}\setminus \{0,-1\}$. If $c$ is a central configuration of the colinear $3$ body problem (corresponding to the potential $V_{3,1}$), then the following equation is satisfied
\begin{equation}\begin{split}\label{eqeuler}
(m_2+m_3)+(2m_2+3m_3)\rho+(3m_3+m_2)\rho^2-\\
(3m_1+m_2)\rho^3-(3m_1+2m_2)\rho^4-(m_1+m_2)\rho^5=0
\end{split}\end{equation}
\end{prop}

In the colinear $3$ body problem, we can always translate a central configuration because the potential is invariant by translation. Moreover, due to this definition, the set of central configurations is also invariant by dilatation, so for any central configuration $q\in\mathbb{C}^3$, after translation and dilatation, we can always write it $q=(-1,0,\rho)$ with $\rho\in\mathbb{C}\setminus \{0,-1\}$. The biggest problem that authors about the subject (see \cite{5}) seem to have encountered is the fact that we have a polynomial of degree $5$, and so not very easy to use. We will see that the complexity of central configuration equations is not a problem at all if we consider the problem differently.

The Theorem \ref{thmmoult} of Moulton suggests that we should work in an opposite way. We fix $\rho>0$ and we search the masses such that $c=(-1,0,\rho)$ is a central configuration. We are then sure that if we consider all possible $\rho$ we will then consider all positive masses (because for each triplet of masses, there is at least one $\rho$ that is convenient). More precisely, we have

\begin{prop}\label{thm0}
The set of masses $m_1,m_2,m_3$ such that $m_1+m_2+m_3=1$ and $c=(-1,0,\rho)$ with
\begin{equation}\label{eq0}
\rho\in \mathbb{C}\setminus \lbrace \rho, \;\;\rho(\rho+1)(1+2\rho+\rho^2+2\rho^3+\rho^4)=0 \rbrace
\end{equation}
is a central configuration, is an affine subspace of dimension $1$ parametrized by
\begin{equation}\begin{split}\label{eq1}
m_1 & =s\\
m_2 & =-\frac{3s\rho^3+3s\rho^4+s\rho^5+s-1+3\rho s-3\rho+3\rho^2 s-3\rho^2}{\rho(1+2\rho+\rho^2+2\rho^3+\rho^4)}\\
m_3 & = \frac{2\rho s+\rho^2s+2s\rho^3+s\rho^4+s-1-2\rho-\rho^2+\rho^3+2\rho^4+\rho^5}{\rho(1+2\rho+\rho^2+2\rho^3+\rho^4)}
\end{split}\end{equation}
Conversely, for each triplet of masses $(m_1,m_2,m_3)\in {\mathbb{R}_+^*}^3,\;\; m_1+m_2+m_3=1$, there exists a central configuration of the form $(-1,0,\rho)$ with condition \eqref{eq0} and $\rho\in\mathbb{R}_+^*$. Eventually, for $\rho \in \mathbb{R},\;\rho\geq 1$, the $m_1,m_2,m_3$ are positive if and only if
$$s \in \left]0,\frac{1+3\rho+3\rho^2}{(1+2\rho+\rho^2+2\rho^3+\rho^4)(1+\rho)}\right[$$
\end{prop}

\begin{proof}
Using equation of Proposition \ref{thmeuler}, we get the following equations
$$\left(\begin{array}{ccc} \!\! 3\rho^3\!-3\rho^4\!-\rho^5  &  1+2\rho+\rho^2\!-\rho^3\!-2\rho^4\!-\rho^5  &  1+3\rho+3\rho^2 \!\!\\1&1&1\\ \end{array}\right) \left(\begin{array}{c} \!\! m_1\!\!\!\\ \!\! m_2\!\!\!\\ \!\! m_3\!\!\!\\ \end{array}\right)=\left(\begin{array}{c} \!\!0\!\!\\\!\!1\!\!\\ \end{array}\right)$$
This is an affine equation and so the space of solutions is an affine subspace. Taking $\rho \in \mathbb{C}\setminus \{ \rho, \;\;\rho(1+2\rho+\rho^2+2\rho^3+\rho^4)\}$, the matrix has always maximal rank, and so the space of solution is of dimension $1$, which we parametrize by $s$. Conversely, the Euler equation \eqref{eqeuler}, thanks to Moulton result for $n=3$, has always exactly one real positive solution.

Eventually, let us look at the case $\rho \in \mathbb{R},\;\rho\geq 1$. We want the masses to be positive, and according to our parametrization, the masses are affine functions in $s$. An affine function changes of sign at most once. Solving $m_i=0$, we get
$$m_1=0\Rightarrow s=0$$
$$m_2=0\Rightarrow s=\frac{1+3\rho+3\rho^2}{3\rho^3+3\rho^4+\rho^5+1+3\rho+3\rho^2}$$
$$m_3=0\Rightarrow s=\frac{1+2\rho+\rho^2-\rho^3-2\rho^4-\rho^5}{2\rho+\rho^2+2\rho^3+\rho^4+1}$$
The last equality gives us for $\rho\geq1$ $s\leq 0$ which is impossible because $m_1\geq 0$. So $m_3$ does not change of sign for any $s> 0$ and is positive. The positivity of $m_2$ gives us the constraint.
\end{proof}

Let us remark that the constraint $\rho \geq 1$ is not a constraint in fact, because using dilatation and the symmetry consisting to reverse the order of \textbf{all the masses}, we exchange $\rho$ by $1/\rho$. After this first proposition, we can study the integrability of the colinear $3$ body problem for real positive masses.

In the following, we will note $W(c) \in M_3(\mathbb{C})$ the $3\times 3$ matrix such that
\begin{equation}\label{eqW}
W(c)_{i,j}=\frac{1} {m_i} \frac{\partial^2} {\partial q_i \partial q_j} V_3(c)
\end{equation}
where $V_3$ is the potential of the colinear $3$ body problem and $c\in\mathbb{C}^3$.

\subsection{Non-integrability}

In this subsection, we will prove Theorem \ref{thmmain1}.

\begin{lem}
For any $\rho \in \mathbb{R},\;\rho\geq 1$, there exists among the masses $(m_1,m_2,m_3)\in{\mathbb{R}_+^*}^3$ such that $m_1+m_2+m_3=1$ and $c=(-1,0,\rho)$ is a central configuration for the triplet of masses $(m_1,m_2,m_3)$ at most $3$ triplet of masses for which the Galois group of first order variational equation has a Galois group whose identity component is abelian.
\end{lem}

\begin{proof}
The matrix $W$ for the central configuration of the form $c=(-\gamma+g,g,\rho \gamma+g)$ is given by
$$\frac{2}{\gamma^3}\left(\begin{array}{ccc}\frac{m_2+3m_2\rho+3m_2\rho^2+m_2\rho^3+m_3}{(1+\rho)^3}&-m_2&-\frac{m_3}{(1+\rho)^3}\\-m_1&\frac{m_1\rho^3+m_3 }{\rho^3}&-\frac{m_3}{\rho^3}\\ -\frac{m_1}{(1+\rho)^3}&-\frac{m_2}{\rho^3}&\frac{m_1\rho^3+m_2+3 m_2\rho+3m_2\rho^2+m_2  \rho^3}{(1+\rho)^3 \rho^3}\\ \end{array}\right)$$
We need to choose $\gamma,g$ such that the multiplier of the central configuration is $-1$ and the center of mass at $0$ (because we want an orbit of the form $c.\phi(t)$). We first compute the spectrum of $W$ which gives
$$\left[0,\frac{4(2\rho^2+3\rho+2)} {(3\rho^3+3\rho^4+\rho^5+1+3\rho+3\rho^2)\gamma^3},-\frac{2(s\rho^4+2s\rho^3-\rho^2+\rho^2s-2\rho+2\rho s-1+s)} {\rho^3(1+2\rho+\rho^2)\gamma^3}\right]$$
where the masses $m_1,m_2,m_3$ are parametrized by $s$ according to the formula \eqref{eq1}. The constraint that the multiplier of $c$ should be equal to $-1$ gives
$$\gamma^3=-\frac{(s\rho^4+2s\rho^3-\rho^2+\rho^2s-2\rho+2\rho s-1+s)} {\rho^3(1+2\rho+\rho^2)}$$
and so we get 
$$Sp(W(c)) =\left\lbrace 0,2,-\frac{4(1+\rho)\rho^3(2\rho^2+3\rho+2)} {(s\rho^4+2s\rho^3-\rho^2+\rho^2s-2\rho+2\rho s-1+s)(1+2\rho+\rho^2+2\rho^3+\rho^4)} \right\rbrace$$
Let us note $G(s,\rho)$ this last eigenvalue, which is a fractional linear function in $s$. The singularity in $s$ of $G$ is at
$$s=\frac{1+2\rho+\rho^2}{1+2\rho+\rho^2+2\rho^3+\rho^4}$$ 
This value of $s$ correspond to the case where the central configuration is in fact an absolute equilibrium. Indeed, we then have the multiplier of the central configuration equal to zero. This special case produce the following set of masses
$$(m_1,m_2,m_3)=\left(\frac{(\rho+1)^2}{1+2\rho+\rho^2+2\rho^3+\rho^4},\frac{-\rho^2}{1+2\rho+\rho^2+2\rho^3+\rho^4},\frac{(\rho+1)^2\rho^2}{1+2\rho+\rho^2+2\rho^3+\rho^4}\right)$$
The mass $m_2$ is always non-positive, and so this case is impossible. Now in the general case, we solve the equation
$$G(s,\rho)\in \left\lbrace \textstyle{\frac{1}{2}}(i-1)(i+2)\;\;i \in \mathbb{N} \right\rbrace$$
and we obtain the following solutions
\begin{align*}
m_1 & =\frac{(\rho+1)(-8\rho^5+k\rho^5-12\rho^4+3k\rho^4-8\rho^3+3k\rho^3+3k\rho^2+3k\rho+k)}{k(1+2\rho^3+\rho^4+2\rho+\rho^2)^2}\\
m_2 & =-\frac{(-8\rho^4+k\rho^4-28\rho^3+2k\rho^3+k\rho^2-40\rho^2-28\rho+2k\rho-8+k)\rho^2}{k(1+2\rho^3+\rho^4+2\rho+\rho^2)^2}\;\;\;\;(E_k)\\
m_3 & =\frac{(\rho+1)(k\rho^5+3k\rho^4+3k\rho^3-8\rho^2+3k\rho^2-12\rho+3k\rho-8+k)\rho^2}{k(1+2\rho^3+\rho^4+2\rho+\rho^2)^2}\\
\end{align*}
with $k \in \lbrace  \frac{1}{2}(i-1)(i+2)\;\; k \in \mathbb{N} \rbrace$. These solutions are not valid for $k=0$, but we have that if $G(s,\rho)=0$ then
$$(1+\rho)(2\rho^2+3\rho+2)=0$$
which is excluded because $\rho\in \mathbb{R}_+^*$. 

Let us look now what happen if we restrict ourselves to positive masses. We take $\rho \geq 1$ and we look at the sign of the masses given by the curves $(E_k)$. We already know according to Proposition \ref{thm0} that the interval $I(\rho)$ to consider for $s$ is the following
$$I(\rho)=\left[0,\frac{1+3\rho+3\rho^2}{(1+2\rho+\rho^2+2\rho^3+\rho^4)(1+\rho)}\right]$$
and noting that $(1+2\rho+\rho^2)>(1+3\rho+3\rho^2)/(1+\rho)$ for $\rho\geq 1$, the singularity of $G(s,\rho)$ is never in $I(\rho)$, and so for $\rho\geq 1$, $G(.,\rho)$ growing on $I(\rho)$.

Then $G(.,\rho)$ is a bijection of $I(\rho)$ on 
$$G(I(\rho),\rho)=\left]\frac{4(1+\rho)\rho^3(2\rho^2+3\rho+2)}{(1+\rho^2+2\rho)(1+2\rho+\rho^2+2\rho^3+\rho^4)},\frac{4(2\rho^2+3\rho+2)(1+\rho)^2}{1+2\rho+\rho^2+2\rho^3+\rho^4}\right[$$
Studying these functions, we prove that the interval $G(I(\rho),\rho)$ is decreasing when $\rho\geq 1$ grows. Knowing that $G(I(1),1)=]2,16[$, the only possible eigenvalues corresponding to a Galois group with an abelian identity component are $5,9,14$.
\end{proof}

Let us now remark that the potential $V_{3,1}$ of the colinear $3$ body problem can be reduced. Indeed, this potential is invariant by translation, and by making the symplectic variable change $p_i\longrightarrow \sqrt{m_i} p_i,\; q_i\longrightarrow q_i/\sqrt{m_i}$, the kinetic part in the Hamiltonian becomes the standard kinetic energy $(p_1^2+p_2^2+p_3^2)/2$. So the set of potential $V_{3,1}$ with parameters $(m_1,m_2,m_3)\in{\mathbb{R}_+^*}^3$, $m_1+m_2+m_3=1$ is a set of homogeneous potentials of degree $-1$ in the plane.

\begin{cor}
The colinear $3$-body problem with positive masses is not meromorphically integrable.
\end{cor}

\begin{proof}
We proved that only the eigenvalues $5,9,14$ are possible for integrability of the colinear $3$-body problem. In \cite{10}, all potentials having these eigenvalues have been classified and there are not meromorphically integrable.
\end{proof}

Remark that in the limit case when two masses tend to $0$, the potential $V_{3,1}$ after reduction is not singular and converges to a potential of the form $\alpha/q_1+\alpha/q_2$, which has the eigenvalue $2$ and is integrable.

\subsection{Higher variational equations}

Let us now compute exactly at which order the variational equations near the unique real Darboux point have a Galois group whose identity component is not Abelian. Indeed, using \cite{10}, we now that on the curves $E_5,E_{14}$, the potentials are integrable at most up to order $4$, and on $E_9$ at most to order $6$ (which reduces to $4$ in our case, because the potential $V_3$ is real and integrable cases to order $5,6$ are complex).

\subsubsection{At order $2$}

To study the Galois group of second order variational equations, we apply Theorem 2 of \cite{9}. We have however to take into account that the kinetic energy is $p_1^2/(2m_1)+p_2^2/(2m_2)+p_3/(2m_3)$ instead of $(p_1^2+p_2^2+p_3^2)/2$. This standard form of kinetic energy can be obtained by a symplectic change of variable. The Hessian matrix $\nabla^2 V(c)$ after this variable change is simply the matrix $W$ defined in \eqref{eqW}.

\begin{lem}
Let $\rho \in \mathbb{R},\;\rho\geq 1$ be a real number, $k\in\{5,9,14\}$ and masses $(m_1,m_2,m_3)\in E_k$. The variational equations at order $2$ near the homothetic orbit associated to $c$ have a Galois group whose identity component is Abelian if and only if the masses belong to the set
$$\left\lbrace \left(\frac{12}{35},\frac{11}{35},\frac{12}{35}\right),\left(\frac{24}{49},\frac{1}{49},\frac{24}{49}\right)\right\rbrace \cup E_9$$
\end{lem}

\begin{proof}
We compute the third order derivatives of $V$ at $c$. Noting $X_2$ the eigenvector of eigenvalue $2$ and $X_3$ the eigenvector of eigenvalue $k$, we have
$$D^3V(X_2,X_2,X_2)=D^3V(X_3,X_3,X_2)=-\frac{3\sqrt{2\rho^2+3\rho+2}\sqrt{2k}}{(\rho+1)^2g^{\frac{4}{3}}\sqrt{k-2}\rho^{\frac{3}{2}}} \quad D^3V(X_3,X_2,X_2)=0$$
$$D^3V(X_3,X_3,X_3)=\frac {-3\sqrt {2}\sqrt {2\,{\rho}^{2}+3\,\rho+2}(\rho-1)P(\rho) }
{(1+2\rho^3+\rho^4+2\rho+\rho^2)^3 \rho^{\frac{3}{2}}g^{\frac{4}{3}}(\rho+1) ^2\sqrt{k(k-2)m_1m_2m_3} }$$
where
$$g={\frac {-4(2\rho^2+3\rho+2)}{ \left( {\rho }^{5}+3\,{\rho }^{4}+3
\,{\rho }^{3}+3\,{\rho }^{2}+3\,\rho +1 \right) k}}$$
$$P(\rho)=(k+10)\rho^6+(5k+50)\rho^5+(8k+120)\rho^4+(7k+158)\rho^3+(8k+120)\rho^2+(5k+50)\rho+k+10$$

According to \cite{9}, the condition for integrability of the second order variational equations are that some of these third order derivative should vanish. Using the table of \cite{9}, the three first third order derivatives never lead to an integrability condition, but the last one does. In particular, for $k=5,14$, the integrability condition is $D^3V(X_3,X_3,X_3)=0$, and there are none for $k=9$.

The only real positive solution of equation $(\rho-1)P(\rho)=0$ for $k=5,14$ is $\rho=1$. Putting this in the parametrization of $(E_5),(E_{14})$, we obtain that the set of possible masses is given by
$$\left\lbrace \left(\frac{12}{35},\frac{11}{35},\frac{12}{35}\right),\left(\frac{24}{49},\frac{1}{49},\frac{24}{49}\right)\right\rbrace \cup E_9$$
\end{proof}

\subsubsection{At order $3$}

Let us now look at order $3$. We will prove that $V_3$ is never integrable at order $3$ near its unique real central configuration.

\begin{lem}
The potential $V_3$ is never integrable at order $3$ at its unique real central configuration.
\end{lem}

\begin{proof}
We will directly use the main Theorem of \cite{90}. After the convenient variable change which send the potential $V_3$ to a planar homogeneous potential with standard kinetic energy, and a rotation dilatation to put the central configuration at $c=(1,0)$, we find that the third order integrability condition can be written
\begin{align*}
-\frac{256}{715}a^2+\frac{13824}{5005}c & =0,\;\; b =0,\;\;k =5\\
-\frac{475136}{57057}a^2-\frac{753664}{101745}b^2+\frac{19759104}{323323}c & =0,\;\;k =9\\
-\frac{2755788800}{7436429}a^2+\frac{19729612800}{7436429}c & =0,\;\; b =0,\;\;k =14
\end{align*}
where the constants $a,b,c$ are
$$a=-\frac{3\sqrt{2\rho^2+3\rho+2}\sqrt{2k}}{(\rho+1)^2g^{\frac{4}{3}}\sqrt{k-2}\rho^{\frac{3}{2}} }, \quad 
b=\frac {-3\sqrt {2}\sqrt {2\,{\rho}^{2}+3\,\rho+2}(\rho-1)P(\rho) }{(1+2\rho^3+\rho^4+2\rho+\rho^2)^3 \rho^{\frac{3}{2}}g^{\frac{4}{3}}(\rho+1) ^2\sqrt{k(k-2)m_1m_2m_3} }$$
$$c=F(\rho,k)$$

%
%
%
%
%


where $F$ is a rational fraction in $\rho,k$ and
$$g=\frac {-4(2\rho^2+3\rho +2)}{ \left( {\rho }^{5}+3\,{\rho }^{4}+3
\,{\rho }^{3}+3\,{\rho }^{2}+3\,\rho +1 \right) k}$$
The constraint $b=0$ for $k=5,14$ comes from order $2$, and we already know the the unique solution is $\rho=1$. The other constraint gives
$$\frac{3024672}{1573}7^{\frac{2}{3}}\neq 0\qquad \frac{2137106227200}{96577}7^{\frac{2}{3}}\neq 0$$
for $k=5,14$ respectively. For $k=9$, the third order integrability constraint is
\begin{align*}
179523957+1436191656\,\rho+5144769684\,{\rho}^{2}+11297844542\,{\rho}^{3}+17938383865\,{\rho}^{4}+\\
23104821764\,{\rho}^{5}+25814403801\,{\rho}^{6}+26361946842\,{\rho}^{7}+25814403801\,{\rho}^{8}+23104821764\,{\rho}^{9}+\\
17938383865\,{\rho}^{10}+11297844542\,{\rho}^{11}+5144769684\,{\rho}^{12}+1436191656\,{\rho}^{13}+179523957\,{\rho}^{14}=0
\end{align*}
This polynomial has no real positive root, and so the constraint is never satisfied.
\end{proof}

\begin{rem}
One could compute the third order integrability condition for any curve $(E_k)$, and even test if this condition could be satisfied thanks to the holonomic approach of third order variational equations in \cite{90}. Here the restriction $(m_1,m_2,m_3)\in {\mathbb{R}_+^*}^3$ is only for physical reasons, but a more complete study is possible. Still remark that this constraint has allowed us to easily bound the eigenvalues, and then to study integrability near the unique real central configuration. If one would allow negative masses, or even complex masses, some results are no longer valid. Especially, there are complex masses which possess a non degenerate central configuration which is integrable at order $3$.
\end{rem}

\begin{figure}
\includegraphics[width=6.5cm]{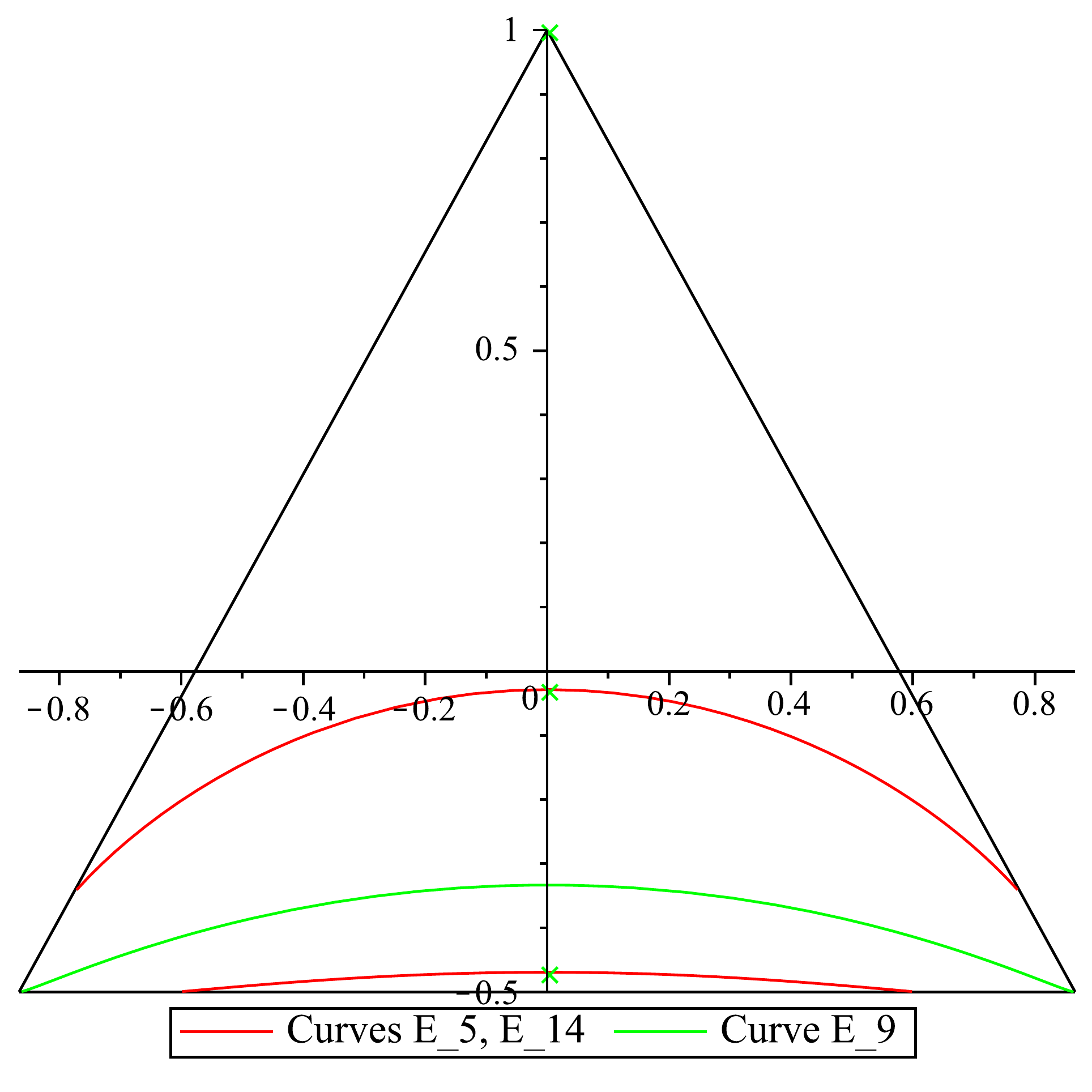}
\includegraphics[width=6.5cm]{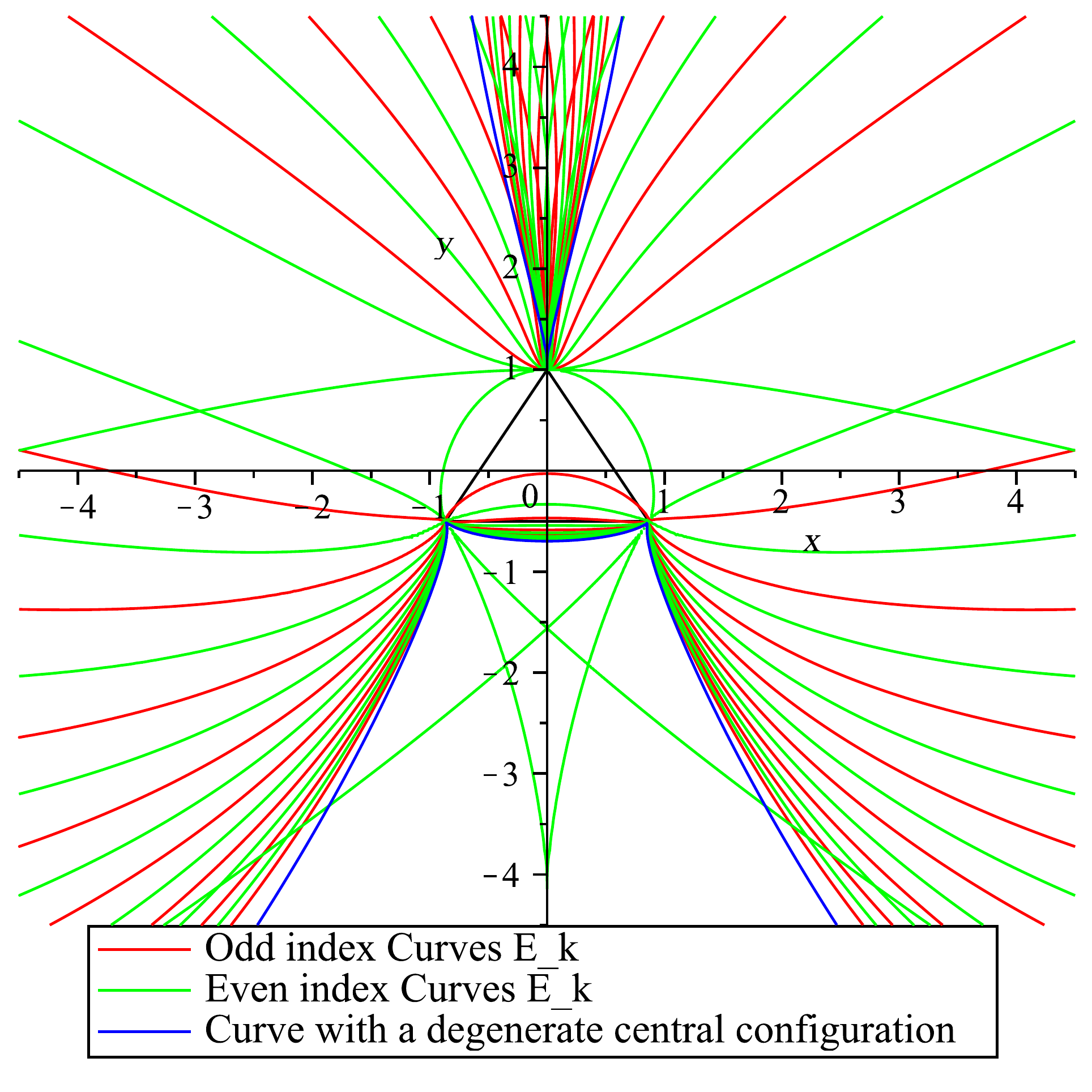}\\
\caption{Graph of the masses having a first order variational equation with a Galois group whose identity component is Abelian. The masses are represented in barycentric coordinates. The masses inside the black triangle are positive. Drawing the curves outside the positive masses reveals that the curves $(E_k)$ accumulate on the curve $(E_{\infty})$. They also intersect on the points $(m_1,m_2,m_3)=(1,0,0),(0,1,0),(0,0,1)$ which are the integrable cases (at the limit when the masses are going to zero through a limiting process).}
\end{figure}


\section{The $4$ body problem}

The previous approach for non integrability proofs can be extented for more complicated systems, as the $4$-body problem, for which a direct approach would be impossible due to the high number of central configurations. The difficulty of the problem of finding these central configurations is famous \cite{93}, thus we will try to need the fewest possible informations on them. The most important quantity are the set of possible eigenvalues of Hessian matrices at the unique real central configuration. In particular, if this set is finite, then the classification approach of \cite{10} is possible.

\subsection{Eigenvalue bounding}

Following the method presented in \cite{10}, we will first try to prove a bound on eigenvalues of the Hessian matrices at Darboux points of $V_{4,1}$. As in \cite{10}, the potential were planar, we need to operate a little differently. Instead of trying to bound directly these eigevalues (whose expression could be domplicated as they appear as roots of the characteristic polynomial), we simply bound the trace of the Hessian matrix. Indeed, the eigevalues of the Hessian matric are of the form $\{0,2,\lambda_1,\lambda_2\}$, and so bounding the trace gives a bound on $\lambda_1+\lambda_2$. Moreover, thanks to Theorem \ref{thmmorales}, we already know that for integrability we must have $\lambda_1,\lambda_2\geq -1$, and thus we get also a bound on $\lambda_1,\lambda_2$.

\begin{thm}
We consider the colinear $4$ body problem with positive masses, whose potential is given by $V_{4,1}$. Let $c$ be the real central configuration (exitence and unicity up to translation due to \ref{thmmoult}) with multiplier $-1$. Let $W\in M_4(\mathbb{C})$ be the matrix
$$W_{i,j}=\frac{1} {m_i} \frac{\partial^2} {\partial q_i \partial q_j} V \quad i,j=1\dots 4$$
Then $tr(W) <70$.
\end{thm}

This value is not the optimal one which has a complicated algebraic expression. Still considering a better bound than this one is not useful as it will not allow us to reduce the number of exceptional cases we will have to deal with.

\begin{proof}
We first remark that after translation, dilatation and changing the order of \textbf{all the masses}, a central configuration of $V_4$ can always be written under the form $c=(-\rho_1,-1,1,\rho_2)$ with $\rho_1\geq \rho_2 > 1$. Moreover, thanks to Moulton Theorem \ref{thmmoult}, we also now that for any fixed positive masses, there always exists a unique central configuration. So we will first fix our central configuration $c=(-\rho_1,-1,1,\rho_2)$, and then compute the masses for which $c$ is a central configuration. Moreover, we will assume that $m_1+m_2+m_3+m_4=1$ because multiplying all the masses by a constant does not change the trace of the matrix $W$.

The equation of central configurations is a linear system in the masses, with $3$ equations for $4$ unknowns. The solution is of the form
$$(m_1,m_2,m_3,m_4)=\left(J_1(\rho_1,\rho_2,m_3),J_2(\rho_1,\rho_2,m_3),J_3(\rho_1,\rho_2,m_3),J_4(\rho_1,\rho_2,m_3)\right) $$
where $J_i$ are rational in $\rho_1,\rho_2$ and affine in $m_3$ (and $J_3(\rho_1,\rho_2,m_3)=m_3$). Now we compute the trace of matrix $W$, and we obtain that $tr(W)$ is also rational in $\rho_1,\rho_2$ and affine in $m_3$.
\begin{lem}
The functions $J_i$ have no singularities for $\rho_1\geq \rho_2 > 1$, and their coefficient in $m_3$ does not vanish for $\rho_1\geq \rho_2 > 1$.
\end{lem}

\begin{proof}
We simply build a polynomial whose factors are the denominators of the functions $J_i$ and numerators of the coefficient in $m_3$ of the functions $J_i$. This polynomial has no real solutions for $\rho_1\geq \rho_2 > 1$.
\end{proof}

So we can handle safely these $J_i$, and solve equations of the form $J_i=0$ in $m_3$ without dealing with singular cases. We need to prove
$$ \max\limits_{J_i>0,\; i=1\dots 4,\;\;\rho_1\geq \rho_2 > 1} tr(W) < 70$$
Let us now remark that for fixed $\rho_1\geq \rho_2 > 1$, the function $tr(W)$ in $m_3$ on the set
$$S_{\rho_1,\rho_2}=\{ m_3\in\mathbb{R} ,J_i(\rho_1,\rho_2,m_3)>0,\; i=1\dots 4\}$$
has its maximum on the boundary of $S_{\rho_1,\rho_2}$ (because $tr(W)$ is affine in $m_3$). So for fixed $\rho_1\geq \rho_2 > 1$, the maximum on the possible $m_3$ has $4$ possible values
$$M_i(\rho_1,\rho_2)= tr(W)(\rho_1,\rho_2,J_i(\rho_1,\rho_2,\cdot)^{-1}(0)) \quad i=1\dots 4$$

Let us now prove the following Lemma

\begin{lem}\label{lem1}
The following bounds hold
$$ M_2(\rho_1,\rho_2)\leq 69.9 \quad  M_3(\rho_1,\rho_2)\leq 69.9 \quad  M_4(\rho_1,\rho_2)\leq 69.9 \qquad \forall\; \rho_1\geq \rho_2 > 1$$
$$ M_1(\rho_1,\rho_2)\leq 69.9 \qquad \forall\; \rho_1\geq \rho_2 > 1,\;\; \rho_1\leq 5$$
\end{lem}

\begin{proof}
These inequalities are automatically proved using RAGlib.
\end{proof}

\begin{lem}\label{lem2}
If $\rho_1\geq 5$, $\rho_1\geq \rho_2 > 1$, then
$$ \max\limits_{m_3\in S_{\rho_1,\rho_2}} tr(W) \in \{M_2(\rho_1,\rho_2), M_3(\rho_1,\rho_2), M_4(\rho_1,\rho_2)\} $$
\end{lem}

\begin{proof}
We set $\rho_1\geq 5$ with $\rho_1\geq \rho_2 > 1$. Assume now that $S_{\rho_1,\rho_2}\neq \emptyset $ and $M_1(\rho_1,\rho_2)$ is the maximum of $tr(W)$ on $S_{\rho_1,\rho_2}$. Then the corresponding masses $(m_1,m_2,m_3,m_4)$ should be all non-negative (recall that the maximum could be reached at the boundary of the domain of positive masses, so for non-negative masses). Solving equation $J_1(\rho_1,\rho_2,m_3)=0$ in $m_3$, we get a rational fraction $D$ in $\rho_1,\rho_2$. We now prove using RAGlib that 
$$D(\rho_1,\rho_2)\leq 0 \qquad \forall \rho_1\geq 5, \rho_1\geq \rho_2 > 1$$
So the only possibility left for having all non-negative masses is that $m_3=D=0$. This implies that $M_1(\rho_1,\rho_2)=M_3(\rho_1,\rho_2)$ and so the Lemma follows.
\end{proof}

Using Lemma \ref{lem2}, we know that if $\rho_1\geq 5$, $\rho_1\geq \rho_2 > 1$, the maximum $M_2$, $M_3$ or $M_4$. These are bounded by $69.9$ thanks to Lemma \ref{lem1}. For $5\geq \rho_1\geq \rho_2> 1$, the maximum of $tr(W)$ can be any of the $M_i$, but due to Lemma \ref{lem1}, all of these are then bounded by $69.9$. So
$$ \max\limits_{J_i>0,\; i=1\dots 4,\;\;\rho_1\geq \rho_2 > 1} tr(W) < 70$$

\end{proof}

\begin{figure}
\includegraphics[width=6.5cm]{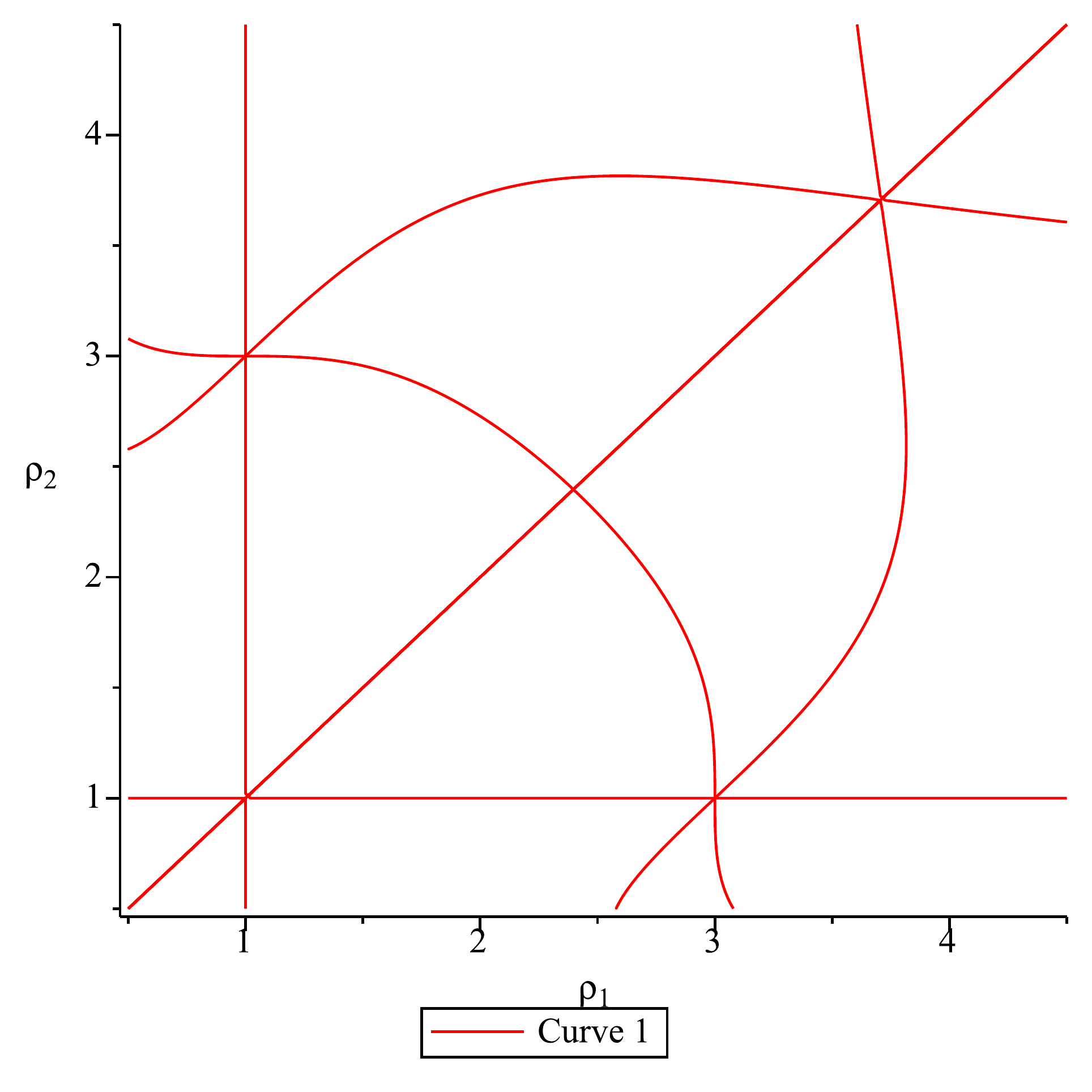}
\caption{Diagramm of the bifurcations between the index $i$ of the $M_i$ that realize the maximum of $tr(W)$. The index $i(\rho_1,\rho_2)$ can only change on one of these curves. Moreover there exists a zone (near $\rho_1,\rho_2=1$) where the set $S_{\rho_1,\rho_2}$ is empty. Numerical analysis gives a maximum around $69.74$.}
\end{figure}

\begin{figure}
\includegraphics[width=6.5cm]{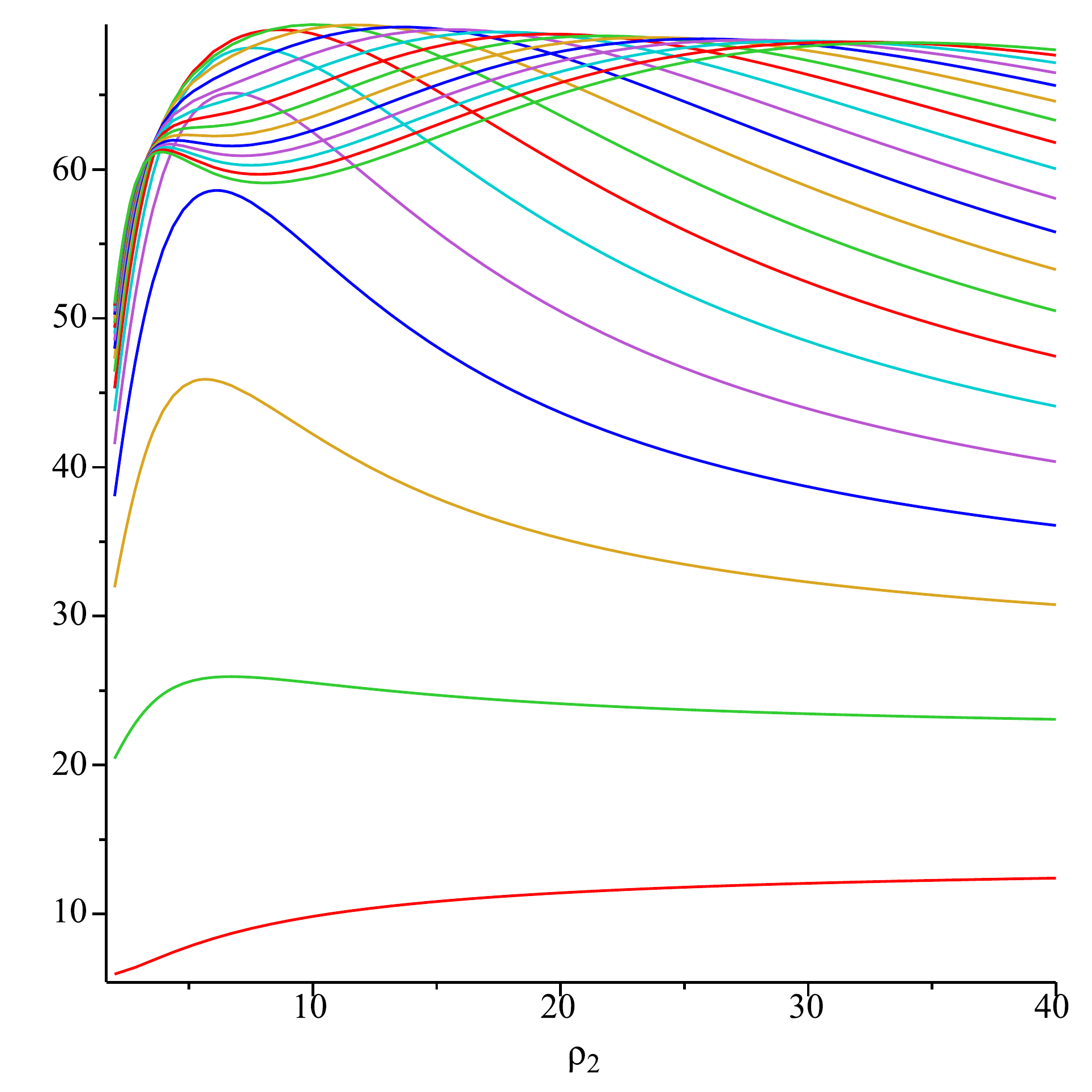}
\includegraphics[width=6.5cm]{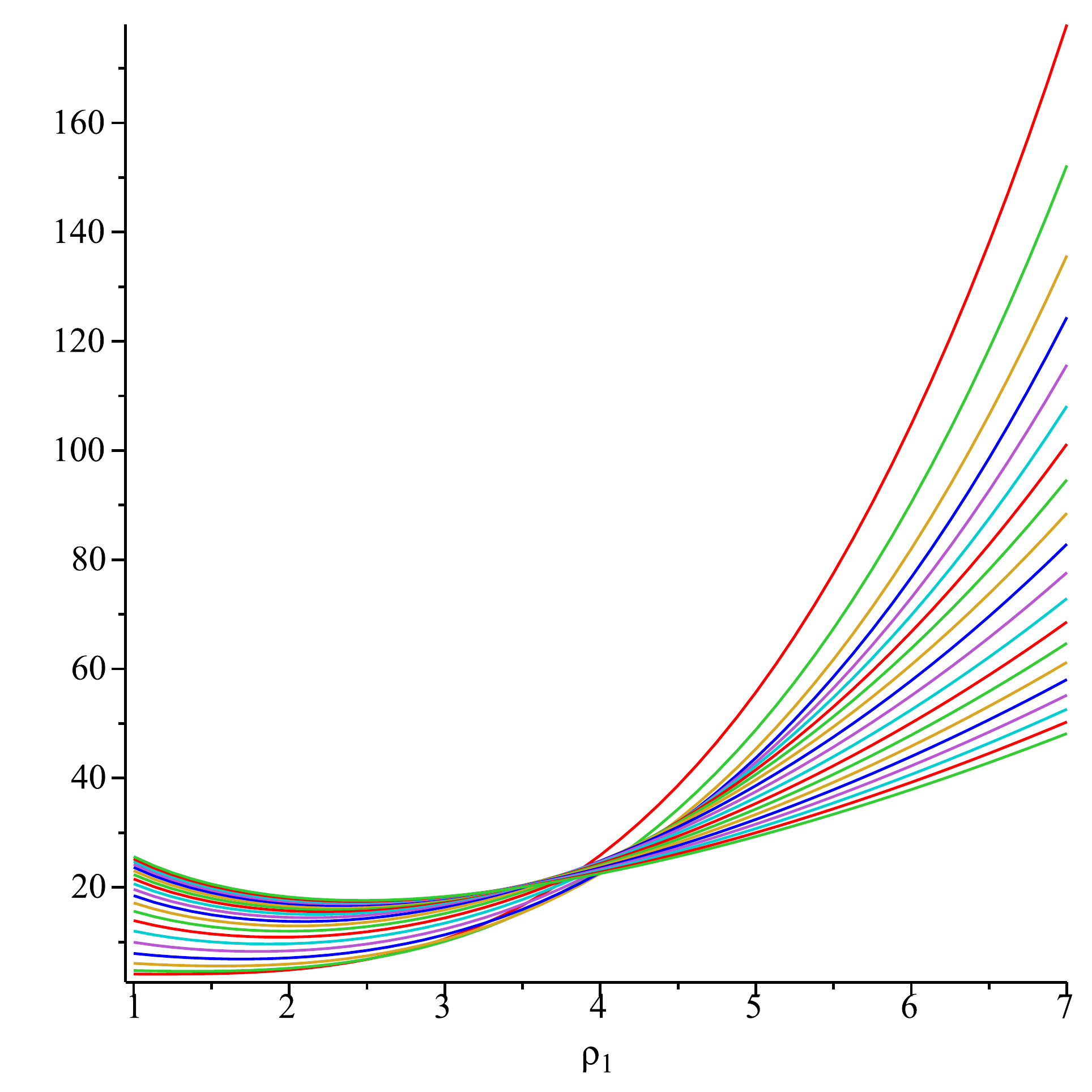}
\caption{Graph of the functions $M_i,i=1\dots 4$. We see that $M_2,M_3,M_4$ are bounded but not $M_1$. This is why we proved that the curve $M_1$ has only to be considered for $\rho_1\leq 5$, allowing to bound the function.}
\end{figure}

\subsection{Symmetric central configurations}

For symmetric central configurations, several cases are possible which are not possible in the non-symmetric case. So we will analyze in this part the case where the real central configuration is of the form $(-\rho,-1,1,\rho)$.

\begin{lem}\label{lemtrace}
The function $tr(W)$ has no singularities for $\rho_1> \rho_2 > 1$, and its coefficient in $m_3$ does not vanish for $\rho_1> \rho_2 > 1$.
\end{lem}

This Lemma is immediately proved by RAGlib. For $\rho_1=\rho_2$, the coefficient in $m_3$ of $tr(W)$ vanishes, making it a special case. On the other hand, this produces an additional symmetry that reduce the number of parameters by $1$ and greatly simplify the formulas

\begin{thm}\label{thmpac} (Pacella \cite{22})
We consider the colinear $4$ body problem potential $V_{4,1}$ with positive masses and the central configuration $c$ with multiplier $-1$ (existence and unicity up to translation due to \ref{thmmoult}). Noting $W\in M_4(\mathbb{C})$ with
$$W_{i,j}=\frac{1} {m_i} \frac{\partial^2} {\partial q_i \partial q_j} V$$
the spectrum of $W$ is of the form $Sp(W)=\{0,2,\lambda_1,\lambda_2\}$ with $\lambda_1,\lambda >2$.
\end{thm}

This already allows to reduce somewhat the possible set of eigenvalues. We will now check if some curves (in the space of masses) corresponding to a couple of eigenvalues $\lambda_1,\lambda_2$ are non-empty for real positive masses.

\begin{lem}
If the potential $V_{4,1}$ with positive masses possess a real central configuration of the form $(-\rho,-1,1,\rho)$, then the spectrum of the Hessian matrix $W$ at the real central configuration with multiplier $-1$ has the form $Sp(W)=\{0,2,\lambda_1,\lambda_2\}$ with
\begin{align*}
\{\lambda_1,\lambda_2\}\in\{\{5,9\},\{5,14\},\{9,27\},\{14,44\} \}
\end{align*}
\end{lem}

\begin{proof}
Using Pacella Theorem, we obtain a better minoration $\lambda_1,\lambda_2 >2$. Knowing that $2+\lambda_1+\lambda_2 <70$, we get the following possibilities
\begin{equation}\begin{split}\label{poss}
\{5,5\},\{5,9\},\{5,14\},\{5,20\},\{5,27\},\{5,35\},\{5,44\},\{5,54\},\{9,9\},\\
\{9,14\},\{9,20\},\{9,27\},\{9,35\},\{9,44\},\{9,54\},\{14,14\},\{14,20\},\{14,27\},\\
\{14,35\},\{14,44\},\{20,20\},\{20,27\},\{20,35\},\{20,44\},\{27,27\},\{27,35\}
\end{split}\end{equation}
We first compute the characteristic polynomial of matrix $W$. Using the same notations as before, the characteristic polynomial has rational coefficients in $\rho_1,\rho_2,m_3$. Factoring it, we put apart the $z(z-2)$ factor (corresponding to eigenvalues $0,2$) and we then get a degree $2$ polynomial $P$ in $z$. The coefficient in $z$ correspond to the trace of $W$, and so is affine in $m_3$. We now put $\rho_1=\rho_2$ in the expression of the characteristic polynomial. The coefficient corresponding to the trace only depends on $\rho_2$. The equation $P(z)=(z-\lambda_1)(z-\lambda_2)$ in $\rho_2,m_3$ give rise to two equations in $\rho_2,m_3$, and we have moreover the constraint of positivity of the masses $m_i$ which can be written in function of $\rho_2,m_3$ with the functions $J_i$. This polynomial system of equations and inequations has real solutions only for $\lambda_1,\lambda_2$ given by the Lemma. 
\end{proof}

\subsection{Reduction of exceptional curves}

In this part, we will always assume that the real central configuration $(-\rho_1,-1,1,\rho_2)$ is such that $\rho_1>\rho_2$.

\begin{lem}
If the potential $V_4$ with positive masses is meromorphically integrable, then the real central configuration $c$ with multiplier $-1$ has a Hessian matrix $W$ with spectrum of the form $Sp(W)=\{0,2,\lambda_1,\lambda_2\}$ with
\begin{align*}
\{\lambda_1,\lambda_2\}\in\{\{5,5\},\{5,9\},\{5,14\},\{5,20\},\{5,27\},\{5,35\},\\
\{5,44\},\{5,54\},\{9,20\},\{9,27\},\{9,35\},\{9,44\},\{9,54\},\{14,44\} \}
\end{align*}
\end{lem}

\begin{proof}
Using Pacella Theorem, we obtain a better minoration $\lambda_1,\lambda_2 >2$. Knowing that $2+\lambda_1+\lambda_2 <70$, we get the possibilities \eqref{poss}. So we only need to eliminate the cases
\begin{align*}
\{9,9\},\{9,14\},\{14,14\},\{14,20\},\{14,27\},\{14,35\},\\
\{20,20\},\{20,27\},\{20,35\},\{20,44\},\{27,27\},\{27,35\}
\end{align*}

We first compute the characteristic polynomial of matrix $W$. Using the same notations as before, the characteristic polynomial has rational coefficients in $\rho_1,\rho_2,m_3$. Factoring it, we put apart the $z(z-2)$ factor (corresponding to eigenvalues $0,2$) and we then get a degree $2$ polynomial $P$ in $z$. The coefficient in $z$ correspond to the trace of $W$, and so is affine in $m_3$. We then solve the equation $tr(W)(\rho_1,\rho_2,m_3)=2+\lambda_1+\lambda_2$ in $m_3$ (using Lemma \ref{lemtrace}, this always produces exactly one solution) and put this solution in $P$. So the only equation we have to study is of the form
\begin{equation}\label{eq5}
Z_0(\rho_1,\rho_2)=P_{\rho_1,\rho_2}(0)-\lambda_1\lambda_2=0 \qquad  \rho_1> \rho_2 > 1
\end{equation}
Using RAGlib, we prove that for $\lambda_1,\lambda_2$ in the upper $12$ cases, this equation has no solutions. This proves the Lemma.
\end{proof}

\begin{rem}
Remark that all the remaining curves are non-empty for $\rho_1\geq \rho_2 > 1$, but this does not imply they are non empty for positive masses (in contrary to the previous part where we have taken into account the positivity of the masses). Numerical evidence suggest that for positive masses, the only possible eigenvalues $\{\lambda_1,\lambda_2\}$ are
$$\{5,9\},\{5,14\},\{5,20\},\{5,27\},\{9,20\},\{9,27\},\{9,35\},\{9,44\},\{14,44\}$$
but taking into account this additional constraint seems too complicated.
\end{rem}

\subsection{Second order variational equations}

Using integrability table of \cite{9}, integrability at second order requires that some of third order derivatives of the potential vanish. Considering only the eigenvalues $\lambda_1,\lambda_2$ (the other ones do not lead to any additional integrability condition) we obtain the following number of conditions (i.e. the number of third order derivatives that should vanish)
\begin{center}
\begin{tabular}{|c|c|}
\hline
$\{5,5\},\{5,14\},\{5,27\},\{14,44\}$& $4$ conditions\\\hline
$\{5,44\},\{5,20\},\{5,35\},\{5,54\}$& $3$ conditions\\\hline
$\{5,9\},\{9,27\},\{9,44\}$  & $2$ conditions\\\hline
$\{9,35\},\{9,54\}$  & $1$ condition\\\hline
$\{9,20\}$ & $0$ condition\\\hline
\end{tabular}\\
\end{center}

The main drawback is that we need a priori to compute the eigenvalues of the Hessian matrix, and due to the parameters, this is quite difficult in our problem. In particular, testing the constraint implies to solve $2$-variables polynomials of degree $172$ and this seems too large to rule out real solutions (if there are none at all). Still in some cases, we can avoid this computation

\begin{prop}\label{propodd}
Let $V$ be a meromorphic homogeneous potential of degree $-1$ in dimension $n$, $c$ a Darboux point of $V$ with multiplier $-1$, and $E$ a stable subspace of $\nabla^2V(c)$. Assume that $\nabla^2V(c)$ is diagonalizable and
\begin{equation}\begin{split}\label{eqcond}
\exists B\subset \mathbb{N} ,\hbox{ with } \max(B) \leq 2\min(B)+1,\;
\hbox{Sp}\left(\left.\nabla^2V(c)\right|_{E}\right) \subset  \{ k(2k+3),\; k\in B \}
\end{split}\end{equation}
If the second order variational equation near the homothetic orbit associated to $c$ has a Galois group whose identity component is Abelian then
$$D^3V(c).(X,Y,Z)=0 \qquad \forall X,Y,Z \in E $$
\end{prop}

\begin{proof}
Using integrability table of \cite{9}, we see that the condition on eigenvalues \eqref{eqcond} implies that the table $A$ for such eigenvalues will only have zeros. So noting $X_1,\dots,X_p$ the eigenvectors associated to eigenvalues $\lambda_i\;\;i=1\dots p$ of $\nabla^2V(c)$, we obtain the integrability condition
$$D^3V(c).(X_i,X_j,X_k)=0 \;\; \forall i,j,k=1 \dots p$$
These $p$ eigenvectors span the invariant subspace $E$, and so by multilinearity, this gives the Proposition.
\end{proof}

We try to avoid to compute the eigenvectors associated to eigenvalues $\{\lambda_1,\lambda_2\}$ for the Hessian matrix a the real central configuartion of $V_4$. In the cases $\{\lambda_1,\lambda_2\}\in\{\{5,5\},\{5,14\},$ $\{5,27\},\{14,44\} \}$, the hypotheses of Proposition \ref{propodd} are satisfied using for $E$ the stable subspace generated by the eigenvectors associated to $\lambda_1,\lambda_2$. And it appears that this subspace is much more easy to compute. Remark also that when the two eigenvalues are equal, then finding the eigenvectors is not necessary as any vector in the corresponding eigenspace is an eigenvector.

\begin{lem}
We consider $V_4$ the potential of the colinear $4$ body problem with positive masses, $c$ the real central configuration with multiplier $-1$, and $W\in M_4(\mathbb{C})$ the matrix such that
$$W_{i,j}=\frac{1} {m_i} \frac{\partial^2} {\partial q_i \partial q_j} V$$
If $Sp(W)=\{0,2,5,5\},\{0,2,5,14\},\{0,2,5,27\},\{0,2,14,44\}$, then the potential $V_4$ is not meromorphically integrable.
\end{lem}

\begin{proof}
We want to consider the sable subspace $E$ of $W$ corresponding to eigenvalues $\lambda_1,\lambda_2$. We already know an eigenvector of eigenvalue $0$, $v=(1,1,1,1)$, and an eigenvector of eigenvalue $2$, the vector $c=(-\rho_1,-1,1,\rho_2)$. As the matrix is symmetric, the eigenspaces are orthogonal, and thus we have $E=\hbox{Span}(v,c)^\bot$. We obtain
$$E=\hbox{Span}((2,-1-\rho_1,\rho_1-1,0),(0,\rho_2-1,-1-\rho_2,2))$$
noting $w_1,w_2$ these two basis vectors of $E$.

Let us first consider the non-symmetric case. As Lemma \ref{lemtrace} applies, we can consider the polynomial $Z_0\in\mathbb{R}[\rho_1,\rho_2]$ given by equation \eqref{eq5}, and
$$ Z_1=D^3 V(c)(w_1,w_1,w_1),\;\; Z_2=D^3 V(c)(w_1,w_1,w_2)$$
$$Z_3=D^3 V(c)(w_1,w_2,w_2),\;\; Z_4=D^3 V(c)(w_2,w_2,w_2)$$
We obtain a system of $5$ equations in two variables (the polynomials $Z_i$ being of degree $58$), and we prove that this system has no solutions for $\rho_1 >\rho_2>1$. Thus the second order variational equation has not a Galois group with an Abelian identity component.

The symmetric case. Only the cases  $Sp(W)=\{0,2,5,14\},\{0,2,14,44\}$ are possible. We have $\rho_1=\rho_2$, and then the condition to have these eigenvalues are of the form of two polynomials in $\rho_2,m_3$. The polynomials $Z_i$ above are still defined, and are polynomials in $\rho_2,m_3$. This system of $6$ equations has no real solutions for $\rho_2>1,m_3>0$, and thus the second order variational equation has not a Galois group with an Abelian identity component.

Thus the potential $V_4$ is not meromorphically integrable in these cases.
\end{proof}

\section{Higher variational equations}

\begin{proof}[Proof of Theorem \ref{thmmain2}]
The still open cases are
\begin{equation}\label{eqlast}
\{5,44\},\{5,20\},\{5,35\},\{5,54\},\{5,9\},\{9,27\},\{9,44\},\{9,35\},\{9,54\},\{9,20\}
\end{equation}
The case $\{9,20\}$ is particularly interesting (and difficult) as there are no integrability conditions at order $2$, and numerical evidence suggest that this case is really possible for positive masses. So this curve gives masses for which all integrability conditions near the unique (up to translation) real Darboux point up to order $2$ are satisfied.

In the same manner as in \cite{10}, we will compute for these remaining sets of eigenvalues higher variational equations. We only need to study real $3$ dimensional homogeneous potentials of degree $-1$. Asuming there exists a real Darboux point $c$, after rotation we can assume that $c=(1,0,0)$ (and the potential is still real). Then the series expansion of $V$ at $c$ will be of the form
\begin{equation}\label{eq2}
V(1+q_1,q_2,q_3)=q_1^{-1}\left(1+\frac{1}{2}\left( \lambda_1 \frac{q_2^2}{q_1^2}+\lambda_2\frac{q_3^2}{q_1^2}\right) +\sum\limits_{i=3}^\infty \sum\limits_{j=0}^i u_{i,j}\frac{q_2^{i-j} q_3^j}{q_1^i} \right)
\end{equation}

As in \cite{10}, the main part of the algorithm consist of finding solutions in $\mathbb{C}(t)\left[\hbox{arctanh}\left( \textstyle{\frac{1}{t}} \right)\right]$ of a large system of linear differential equations, which are the $k$-th variational equation. This $k$-th variational equations is put under block triangular form to make computation faster. Only the last equation is solved through the variation of parameters technique and then its monodromy analyzed through commutativity condition of monodromy in \cite{9}.

Instead of computing a basis of solutions, we only compute several solutions, that through empirical evidence, will lead to the strongest integrability conditions. The output of the algorithm is a set of polynomial conditions on higher order derivatives of the potential $V$ at the Darboux point $c$, so here polynomial conditions on the $u_{i,j}$. As presented in \cite{10}, if a non degeneracy type condition is satisfied (see \cite{10} Definition 4.1), we will be able to express higher order derivative in function of lower order ones. In our cases, this will always be the case for variational equations of order $\geq 3$ (but we are lucky, because it seems that if eigenvalues are spaced enough, degeneracy at any order is possible). This allows us in particular to express all derivatives of order $\geq 4$ in function of $u_{3,0},u_{3,1},u_{3,2},u_{3,3}$. The possible series expansions are written in Appendix A. Variational equations up to order $4$ have been analyzed. Still, at order $4$, some combinations of eigenvalues are still possible, and thus looking at order $5$ is necessary. However, a speed-up is possible in certain cases:

\subsection{An invariant subspace of the $5$-th order variational equation}

\begin{lem}\label{leminv}
Let $V$ be a real meromorphic homogeneous potential of degree $-1$ in dimension $3$. Assume that $V$ has a series expansion of the form
$$V(1+q_1,q_2,q_3)=q_1^{-1}\left(1+\frac{1}{2}\left( \lambda_1 \frac{q_2^2}{q_1^2}+\lambda_2\frac{q_3^2}{q_1^2}\right) +\sum\limits_{i=3}^\infty \sum\limits_{j=0}^i u_{i,j}\frac{q_2^{i-j} q_3^j}{q_1^i} \right)$$
with $u_{3,1}=u_{4,1}=u_{5,1}=0$ and $\lambda_1\in \{5,9,14,20\}$. Then $V$ is not meromorphically integrable.
\end{lem}

\begin{proof}
The dynamical system associated to $V$ is of the form $\ddot{q}=\nabla V(q)$. Let us compute $\partial_{q_3} V$
$$\partial_{q_3} V=q_1^{-1}\left(\lambda_2\frac{q_3}{q_1^2} +\sum\limits_{i=3}^5 \sum\limits_{j=2}^i ju_{i,j}\frac{q_2^{i-j} q_3^{j-1}}{q_1^i} +\sum\limits_{i=6}^\infty \sum\limits_{j=1}^i ju_{i,j}\frac{q_2^{i-j} q_3^{j-1}}{q_1^i} \right) $$
Thus we get that the series expansion of $\partial_{q_3} V$ at order $5$ for $q_3=0$ is
$$\partial_{q_3} V=u_{6,1}\frac{q_2^5}{q_1^7}+ O\left(\left(q_2,q_3\right)^6/q_1^8\right)$$
As we see, there is only one term left, and it is of order $5$. Let us now look at the $5$-th order variational equation.

This variational equation will have an invariant subspace $\mathcal{W}$ corresponding to the $5$-th order variational equation of $\tilde{V}$, the restriction of $V$ to the plane $q_3=0$. Let us now look the variational equation on $\mathcal{W}$. The potential $\tilde{V}$ is a $2$-dimensional homogeneous potential of degree $-1$, and it has a Darboux point at $(1,0)$. The eigenvalues of the Hessian matrix of $\tilde{V}$ at this point are $\{2,\lambda_1\}$. Now using \cite{10}, we know that for any choice of \textbf{real} $\tilde{V}$ with $\lambda_1\in \{5,9,14,20\}$, the $5$-th order variational equation has not a virtually Abelian Galois group. Thus the $5$-th order variational equation of $V$ has not a virtually Abelian Galois group, and thus $V$ is not meromorphically integrable.
\end{proof}

\begin{rem}
Remark that physically, the condition $u_{3,1}=u_{4,1}=u_{5,1}=0$ implies that the plane $q_3=0$ is invariant at order $4$. At order $5$, it is no more invariant, still the derivatives in time of $q_1,q_2$ do not depend on $q_3$.
\end{rem}

We now use Lemma \ref{leminv}. Looking at the series expansions in Appendix A we have computed, we see that for all of them except the last one, we have either $u_{3,1}=u_{4,1}=u_{5,1}=0$ or $u_{3,2}=u_{4,3}=u_{5,4}=0$ (or both). In the first case we can apply directly Lemma \ref{leminv}. In the second case, we just have to exchange $q_2,q_3$, and the hypotheses of Lemma \ref{leminv} are satisfied. So except for the case $(\lambda_1,\lambda_2)=(9,20)$, the hypotheses are satisfied and thus no real meromorphic homogeneous potential of degree $-1$ in dimension $3$ with $c=(1,0,0)$ as a Darboux point of $V$ with multiplier $-1$ and these pairs of eigenvalues are meromorphically integrable.

\subsection{The case $\{9,20\}$}

In the last subsection, we tried to avoid to compute the Galois group of the $5$-th order variational equation as it is computationally expensive. At order $4$, the ideal $\mathcal{I}_4$ is zero-dimensional, but still has real solutions (given in Appendix A). As we can see in Appendix A, the previous Lemma does not apply for these eigenvalues. Thus it is necessary to compute completely the $5$-th order variational equation. The coefficients of the series expansion at order $4$ are polynomials in $u_{3,0},u_{3,1},u_{3,2},u_{3,3}$ modulo the ideal $\mathcal{I}_4$. As the algorithm never needs to inverse an element of this ring (which contains zero divisors), it also works at order $5$. The output (after one week of computation) is the ideal $\mathcal{I}_5$, which happen to be improper. Thus $\mathcal{I}_5=<\! 1\!>$. We deduce then

\begin{lem}
Let $V$ be a real meromorphic homogeneous potential of degree $-1$ in dimension $3$. Assume that $c=(1,0,0)$ is a Darboux point of $V$ with multiplier $-1$ and $\hbox{Sp}(\nabla^2 V(c))=\{2,9,20\}$. Then $V$ is not meromorphically integrable.
\end{lem}

\end{proof}

\section{The planar $n$-body problem}

Let us prove in this section Theorem \ref{thmmain3}. The main tool will be the following Theorem

\begin{thm} (Pacella \cite{22} Theorem 3.1)
We consider the colinear $n$ body problem with positive masses and $c$ a configuartion with multiplier $-1$, given by potential $V_{n,2}$. Noting $W\in M_n(\mathbb{C})$ with
$$W_{i,j}=\frac{1} {m_i} \frac{\partial^2} {\partial q_i \partial q_j} V_{n,2}$$
the spectrum of $W$ is of the form $Sp(W)=\{0,2,\lambda_1,\dots,\lambda_{n-2}\}$ with $\lambda_i>2,\;i=1\dots n-2$.
\end{thm}

\begin{proof}
For the $n$ body problem in the plane, the Hessian matrix to compute is of the form $W\in M_{2n}(\mathbb{C})$
$$W_{i,j}=\frac{1} {m_i} \frac{\partial^2} {\partial q_i \partial q_j} V_{n,2}$$
with the notation $m_{i+n}=m_i$. Computing this matrix at a colinear central configuration, we obtain a matrix of the form
$$W=\left(\begin{array}{cc} A&0\\ 0& -\frac{1}{2} A \\ \end{array}\right)$$
Due to Pacella Theorem, we have moreover that $Sp(A)=\{0,2,\lambda_1,\dots,\lambda_{n-2}\}$ with $\lambda_i>2,\;i=1\dots n-2$. Then the spectrum of $W$ is of the form
$$\{0,2,\lambda_1,\dots,\lambda_{n-2},0,-1,-\textstyle{\frac{1}{2}}\lambda_1,\dots,-\textstyle{\frac{1}{2}}\lambda_{n-2} \}$$
According to integrability condition of the Morales-Ramis Theorem \ref{thmmorales}, all allowed eigenvalues for integrability are greater or equal to $-1$. These conditions cannot be satisfied as $-\textstyle{\frac{1}{2}}\lambda_i< -1,\;i=1\dots n-2$. Thus the planar $n$ body problem with positive masses is not meromorphically integrable.
\end{proof}

\section{Integrable $n$-body problems}

In this section, we will progress in the opposite way. Instead of trying to prove non integrability, we come from already known integrable cases, and we try to know if after some transformations, they correspond to particular cases of the $n$ body problem.

\begin{prop}\label{prop1}
The potential $V$ in $n$ variables $q_1,\dots,q_n$
\begin{equation}\label{potint}
V(q)=\sum\limits_{l=1}^p a_l\left(\sum\limits_{j=j_l+1}^{j_{l+1}} q_j^2\right)^{-1/2}
\end{equation}
with $0=j_1<j_2<\dots <j_{p+1}\leq n,\; a_l\in\mathbb{C}$ is integrable in the Liouville sense. For any complex orthogonal matrix $R\in \mathbb{O}_n(\mathbb{C})$, the potential $V(Rq)$ is integrable in the Liouville sense.
\end{prop}

Here the kinetic part is assumed to be $T(p)=\lVert p \lVert^2/2$ and so the potential $V$ is associated to a Hamiltonian system with $H(p,q)=T(p)+V(q)$. 

\begin{proof}
The potential $V$ of equation \eqref{potint} is a decoupled linear combination
$$V(q)= \sum\limits_{l=1}^p a_l V_l(q_{j_l+1},\dots, q_{j_{l+1}}),\quad\;\; V_l(q_{j_l+1},\dots, q_{j_{l+1}})=\left(\sum\limits_{j=j_l+1}^{j_{l+1}} q_j^2\right)^{-1/2}$$
These potentials are invariant by the rotation group $\mathbb{O}_{j_{l+1}-j_l}(\mathbb{C})$ and so are integrable. Thus the potential $V$ is integrable. As integrability is preserved by any orthogonal transformation, the potential $V(Rq)$ will also be integrable in the Liouville sense.
\end{proof}

Although these potentials seem to have a quite simple expression, the orthogonal transformation $R$ can mix the variables (the decomposition of $V$ is not necessarily conserved). However, the potential can always be written
\begin{equation}\label{eq00}
V(q)=\sum\limits_{l=1}^p a_lQ_l(q)^{-1/2}
\end{equation}
with $Q_i$ quadratic forms. And as an orthogonal transformation conserves the rank of these quadratic forms, we have moreover $\sum_{l=1}^p \hbox{rank}\;Q_l \leq n$.

The Hamiltonian of the $n$ body problem in dimension $d$ can be written
$$H(p,q)=\sum\limits_{i=1}^n \frac{\lVert p_i\lVert^2}{2m_i} +\sum\limits_{i<j} m_im_j\left(\sum\limits_{k=1}^d (q_{i,k}-q_{j,k})^2\right)^{-1/2} $$
The kinetic part of $H_{n,d}$ is not $\lVert p\lVert^2/2$ (as in Proposition \ref{prop1}). To transform the kinetic part to the standard one, we only have to make the variable change $q_{i,k}\mapsto q_{i,k}/\sqrt{m_i}$. The potential now becomes
$$\tilde{V}_{n,d}=\sum\limits_{i<j} m_im_j\left(\sum\limits_{k=1}^d (q_{i,k}/\sqrt{m_i}-q_{j,k}/\sqrt{m_j})^2\right)^{-1/2}$$
This expression is similar to equation \eqref{eq00}, but there could be too many quadratic forms. After reduction by translation, the potential becomes a $(n-1)d$-dimensional potential. There are $n(n-1)/2$ independent quadratic forms, and to be of the form \eqref{eq00}, we need $n(n-1)d/2\leq (n-1)d$, which implies $n\leq 2$. So in general, the potential $\tilde{V}_{n,d}$ is not of the form \eqref{potint}, but it could be for some restricted cases.

\begin{defi}
We say that a vector space $W\subset \mathbb{R}^{nd}$ is an invariant vector space if  
$$\forall q\in W,\;\;\; \nabla V_{n,d} \in W$$
\end{defi}

This definition generalizes central configurations, which correspond to the case $\dim  W=1$. Needless to say, as it is more difficult to find the invariant vector spaces than finding central configurations, we will not try to be exhaustive in this search. Let us remark that we already know some invariant vector spaces as the isosceles $3$-body problem, the collinear $3$-body problem (which is an invariant vector space of the planar $3$ body problem). Several others can be found using symmetries.

Let us now establish some rules to find vector spaces $W$ and masses $m$ such that $\left.V\right|_W$ is of the form \eqref{potint} up to an orthogonal transformation. A necessary condition is that it can be written under the form \eqref{eq00}. So in the expression of $\left.V\right|_W$ we should try to have the lowest possible number of independent quadratic forms (corresponding to mutual distances) with the lowest possible rank.

Remark that if we allow negative masses, some terms in the sum in $V_{n,d}$ could cancel each other, thus reducing greatly the number of quadratic forms. So it seems that finding examples will be easier when negative masses are allowed. And indeed, all interesting examples we will find require a negative mass. Let us now prove Theorem \ref{thmmain4}.

\subsection{An integrable $5$ body problem}

\begin{proof}
The vector space $W$ is of dimension $10-6=4$. The mass $-1/4$ is at the origin which is the center of mass of the system. On the vector space $W$, the $2$-nd and the $4$-th body are symmetric in respect to the origin, as well as the $3$-th and $5$-th bodies. Due to these symmetries, the vector space $W$ is invariant. We can thus restrict our potential to $W$. Now computing the potential $V$ on $W$ we find $\left.V_{5,2}\right|_W=$
\begin{equation}\label{potres1}
\left((q_{2,1}-q_{3,1})^2+(q_{2,2}-q_{3,2})^2\right)^{-1/2} +\left((q_{2,1}+q_{3,1})^2+(q_{2,2}+q_{3,2})^2\right)^{-1/2}
\end{equation}
There are only two quadratic forms with rank two for each. As $2+2=4=\dim W$, we are under the form \eqref{eq00}. We can now try to put this potential under the form \eqref{potint}. This is done by the following orthogonal transformation
$$R=\frac{1}{\sqrt{2}} \left(\begin{array}{cccc} 1&-1&0&0\\0&0&1&-1\\1&1&0&0\\0&0&1&1 \end{array}\right)$$
acting on $q_{2,1},q_{3,1},q_{2,2},q_{3,2}$ in this order. Thus $\left.V_{5,2}\right|_W$ is integrable.
\end{proof}

On $W$, the bodies are always on the edges of a parallelogram whose center is the origin (where lies the mass $-1$). Looking at the forces acting on the bodies, we see that they are not attracted by the center at all (because the repulsion of the central mass $-1/4$ exactly compensates the attraction of the opposite masses $1$ at twice the distance). The masses are then only attracted by their neighbours. Looking at the expression of the potential \eqref{potres1}, we see that the force acting on the center of vertices of the parallelogram (which are $(\pm q_{2,1}\pm q_{3,1},\pm q_{2,2}\pm q_{3,2})$) is toward the center. Thus the motion of these centers are conics with focus at the origin.\\

\newpage

\begin{minipage}{5.1cm}
\centering
\includegraphics[width=4.9cm]{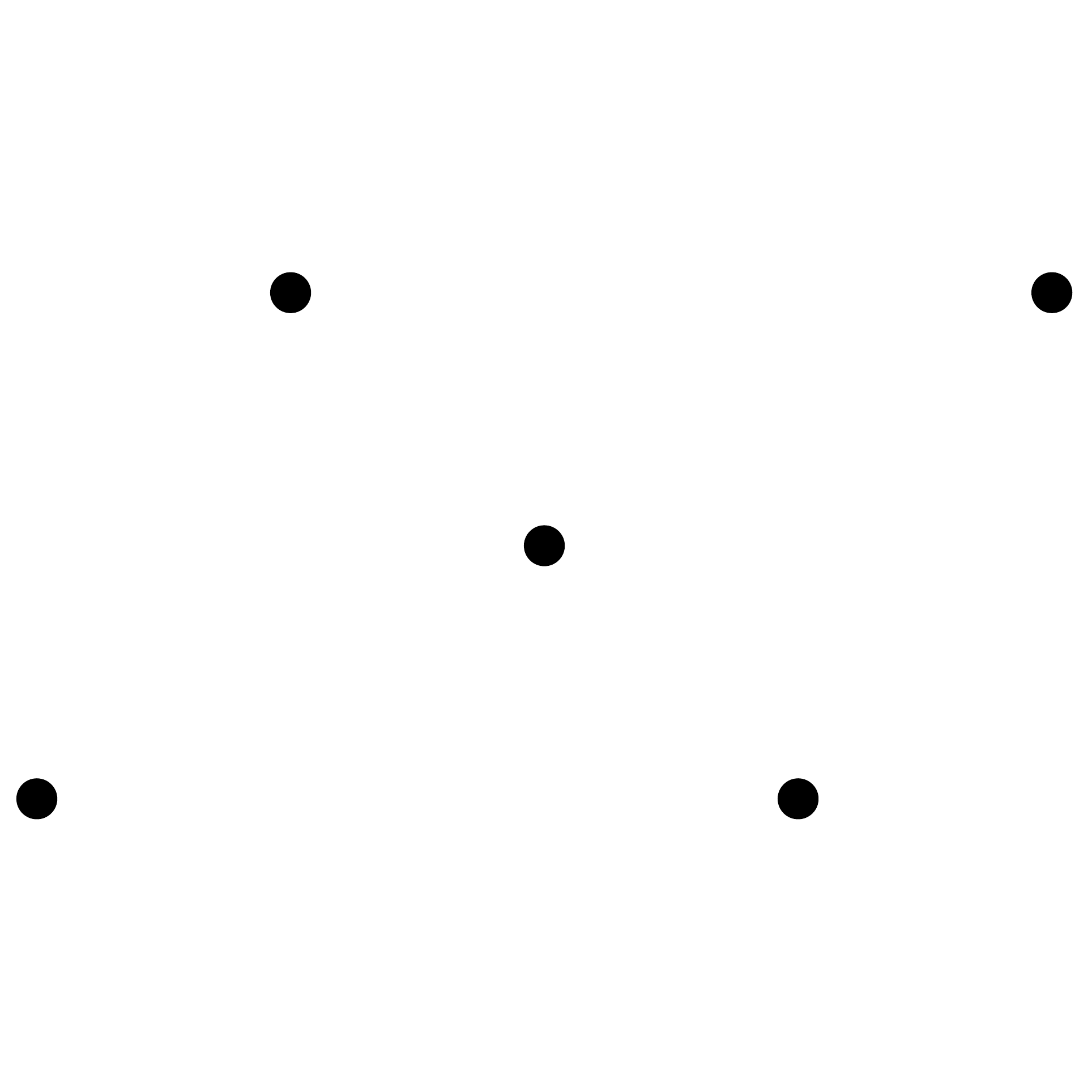}\\
\end{minipage}
\hfill
\begin{minipage}{6cm}
\centering
\includegraphics[width=5.8cm]{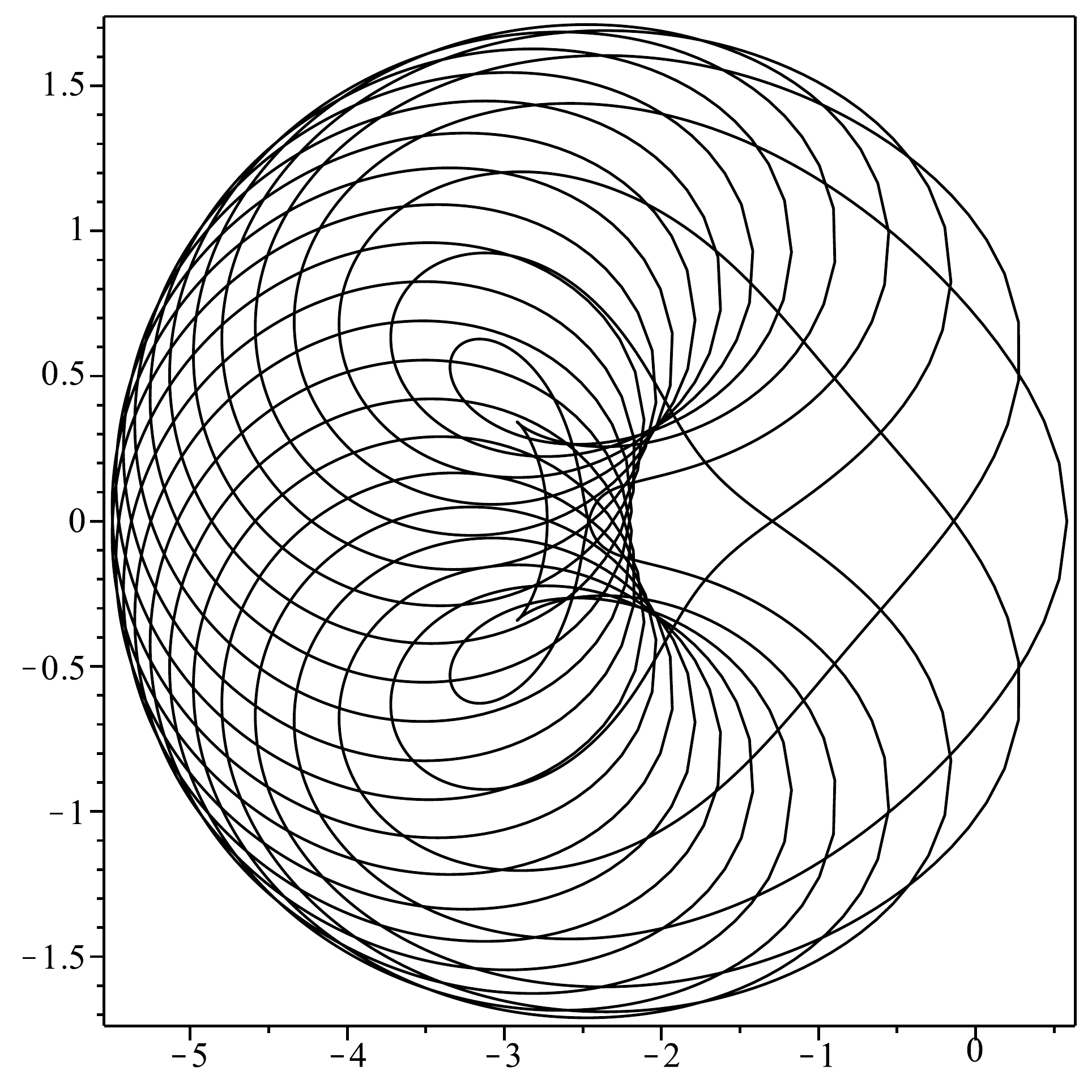}\\
\end{minipage}
\begin{minipage}{6cm}
\centering
\includegraphics[width=5.8cm]{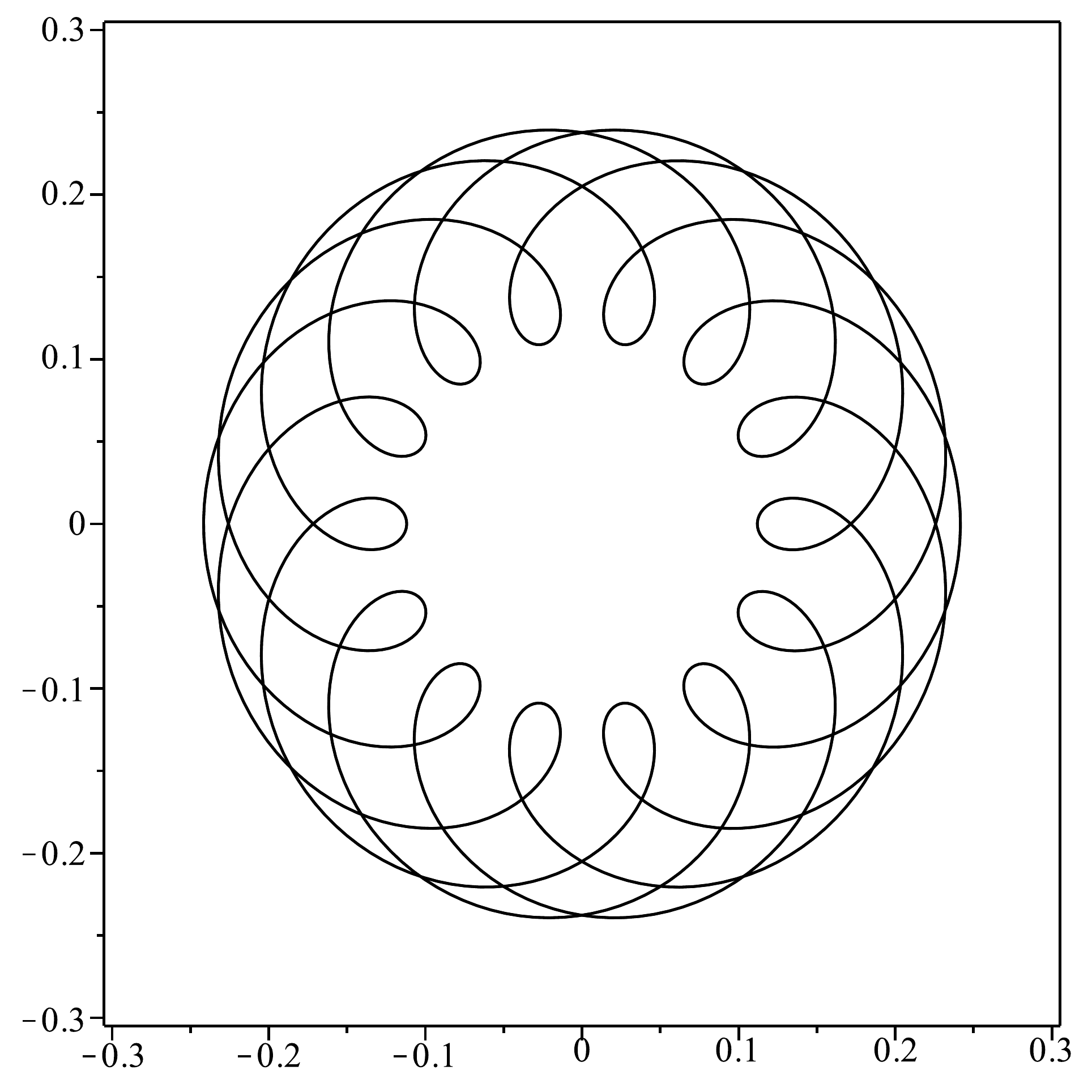}\\
\end{minipage}
\hfill
\begin{minipage}{6cm}
\centering
\includegraphics[width=5.8cm]{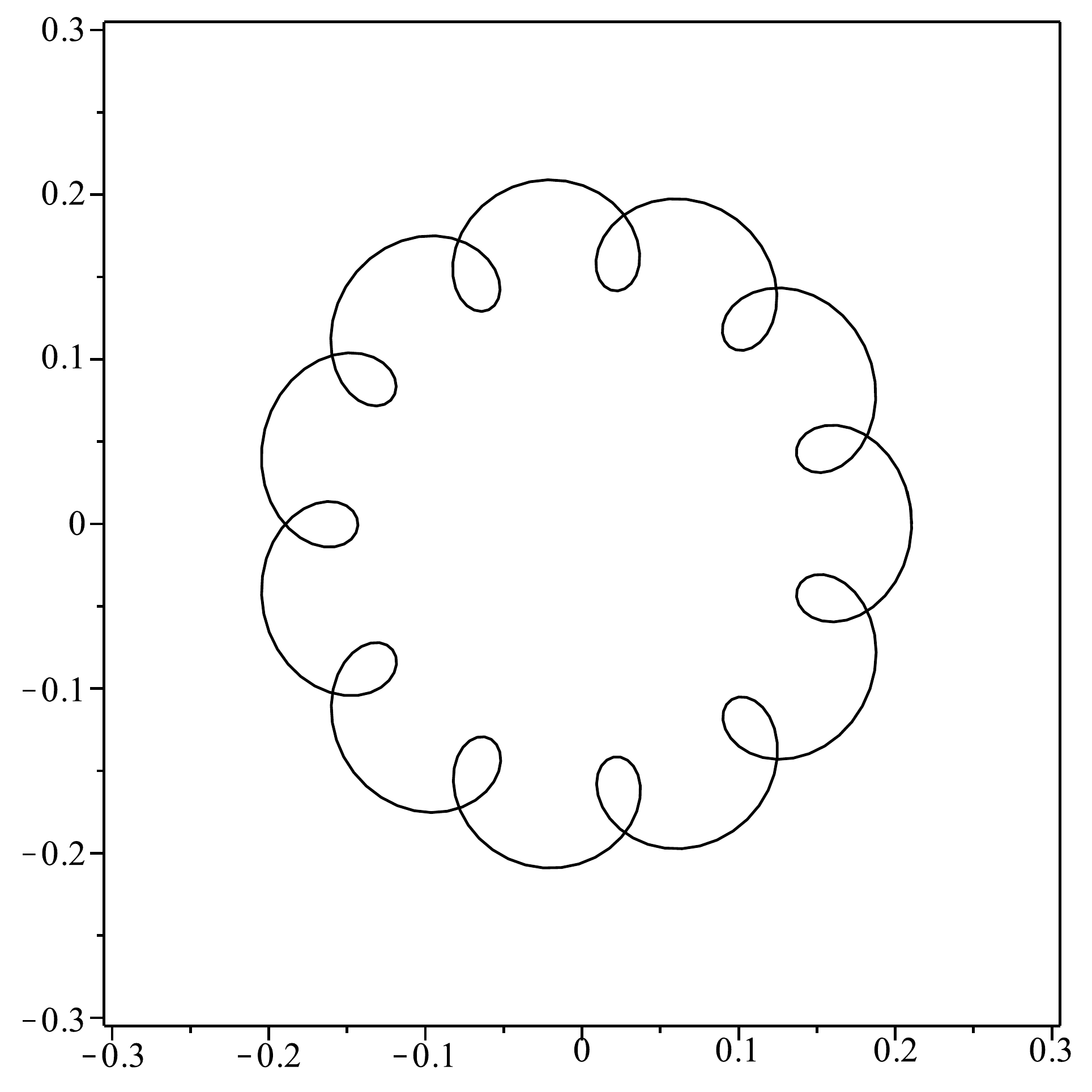}\\
\end{minipage}\\
\begin{center}
\textit{The configuration of the $5$ bodies and examples of motions of a body of mass $1$ with ellipses with rational period ratio.}
\end{center}

So the motion of a body of mass $1$ is the composition of to conic motions. The body has a conic motion whose focus is the center of mass of two bodies of mass $1$, and this center of mass has a conic motion with focus at the origin. If the two conics are ellipses with rational period ratio, this leads to (algebraic) choreographies of the two bodies.

\subsection{An integrable $n+3$ body problem}

\begin{proof}
The space $W$ is of dimension $3$. The forces between the $n$ cocyclic masses and the central mass exactly compensate. The forces between the $3$ last masses also compensate (as this is also an absolute equilibrium). So the only forces between the bodies are between the last two masses and the cocyclic masses. But due to symmetry and the fact that the masses are cocyclic, this force only involves one distance. Thus the potential is of the form
$$V=\gamma\left(q_{1,1}^2+q_{1,2}^2+q_{n+2,3}^2\right)^{-1/2}$$
This potential corresponds to a central force, and thus is integrable.
\end{proof}

Let us look at an example. The most known cyclic central configuration is the regular polygon. We have $m_1=\dots=m_n=1$. The central mass (chosen to produce an absolute equilibrium) and the potential are then
$$-\alpha=-\frac{1}{2}\sum\limits_{k=1}^{n-1} \sin\left(\frac{k\pi}{n}\right)^{-1}\qquad V=4n\alpha\left(q_{1,1}^2+q_{1,2}^2+q_{n+2,3}^2\right)^{-1/2}$$
The motion is the following: the $n$ bodies describe conics in the plane, and the two symmetrical last bodies move along the vertical line. Remark that the motion of the bodies on the vertical line is not determined by the motion of the bodies in the plane (this is not a rigid motion as in the case of central configurations). This vertical motion depends on the ``inclination'' of the conic orbit chosen for the above potential.

\begin{minipage}{5cm}
\centering
\includegraphics[width=4.5cm]{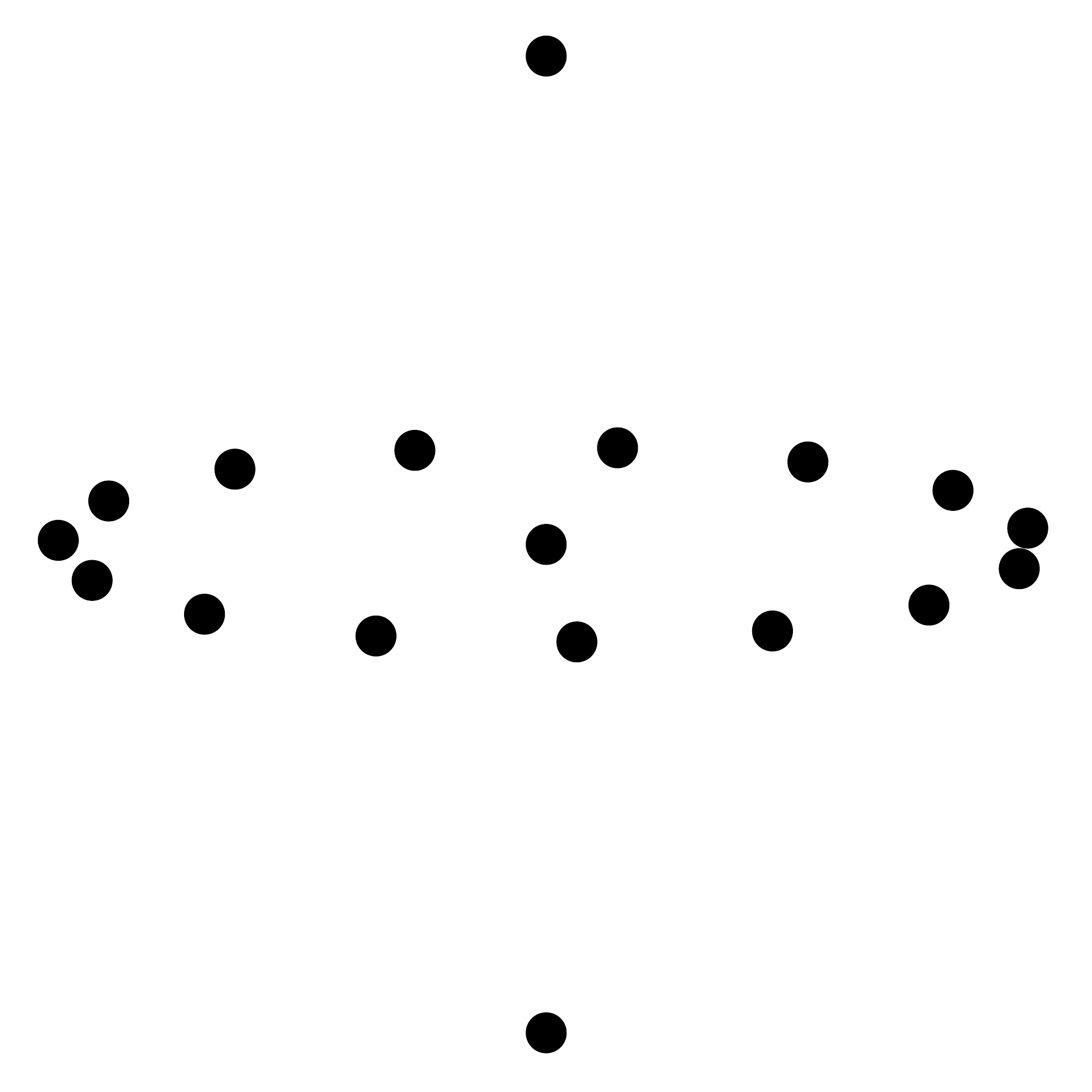}\\
\end{minipage}
\hfill
\begin{minipage}{7cm}
\centering
\includegraphics[width=6.5cm]{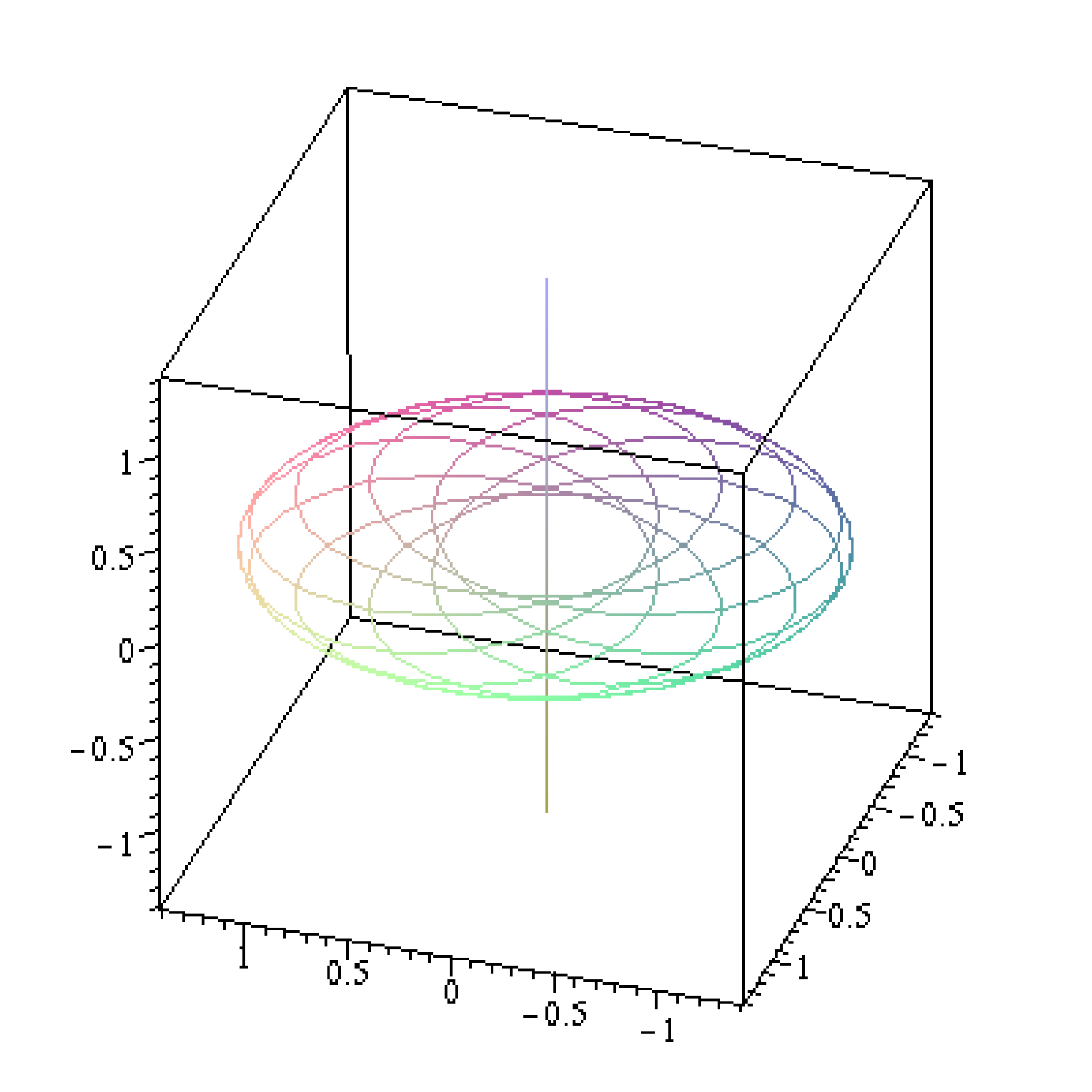}\\
\end{minipage}
\textit{A configuration of the $n+3$ bodies with a regular polygon and an example of motion of the bodies with an ellipse with non-zero inclination.}

\appendix
\section{Integrable series expansions at order $4$}

The results are written in the following way. We give a series expansion of the form \eqref{eq2} of $V$, such that the $k$-th order variational equation of $V$ near $c$ has a virtually Abelian Galois group if and only if $(u_{3,0},u_{3,1},u_{3,2},u_{3,3})\in \mathcal{I}_k^{-1}(0)$. The sequence of ideals $\mathcal{I}_k$ is growing, and we compute these conditions up to order $4$. For the eigenvalues in \eqref{eqlast}, thay are geven below. Remark that the Hilbert dimension of the ideals $\mathcal{I}_4$ greatly depend on eigenvalues, and that sometimes exceptional possible solutions appear in the $4$-th order variational equation. In particular, the restriction of these series expansions to the planes in $(q_1,q_2)$ and $(q_1,q_3)$ does not always lead to integrable series expansion at order $4$ on these planes. 

\begin{center}
\begin{turn}{-90}
\hspace{-0.5cm}
\begin{minipage}{22.5cm}
\begin{tiny}
$q_1^{-1}\left(1+\frac{5}{2}\frac{q_2^2}{q_1^2}+\frac{44}{2}\frac{q_3^2}{q_1^2}+u_{3,1}\frac{q_2^2q_3}{q_1^3}+
\left(\frac{7}{459}u_{3,1}^2+\frac{175}{24}\right)\frac{q_2^4}{q_1^4}+  \left(-\frac{734}{7775}u_{3,1}^2+\frac{130779}{622}\right)\frac{q_2^2q_3^2}{q_1^4}+
\frac{682538736}{1082611}\frac{q_3^4}{q_1^4}+ O\left( \frac{(q_1,q_2)^6}{q_1^6} \right)\right)\;\;
\mathcal{I}_2=\mathcal{I}_3=<\!u_{3,0},u_{3,2},u_{3,3}\!>,\; \mathcal{I}_4=<\!u_{3,0},u_{3,1},u_{3,2},u_{3,3}\!>$\\
$q_1^{-1}\left(1+\frac{5}{2}\frac{q_2^2}{q_1^2}+\frac{20}{2}\frac{q_3^2}{q_1^2}+u_{3,3}\frac{q_3^3}{q_1^3}+
\frac{175}{24}\frac{q_2^4}{q_1^4}+ \frac{13575}{142}\frac{q_2^2q_3^2}{q_1^4}+ 
\left(\frac{105347}{840700}u_{3,3}^2+\frac{153200}{1201} \right)\frac{q_3^4}{q_1^4}+ O\left( \frac{(q_1,q_2)^6}{q_1^6} \right)\right)\qquad \mathcal{I}_2=\mathcal{I}_3=<\!u_{3,0},u_{3,1},u_{3,2}\!> \quad \mathcal{I}_4=<\!u_{3,0},u_{3,1},u_{3,2},u_{3,3}\!>$\\
$q_1^{-1}\left(1+ \frac{5}{2}\frac{q_2^2}{q_1^2}+\frac{35}{2}\frac{q_3^2}{q_1^2}+u_{3,3}\frac{q_3^3}{q_1^3}+
\frac{175}{24}\frac{q_2^4}{q_1^4}+ \frac{13575}{142}\frac{q_2^2q_3^2}{q_1^4}+
\left(\frac{108970043}{1516693800}u_{3,3}^2+ \frac{8028474825}{20222584} \right)\frac{q_3^4}{q_1^4}+ O\left( \frac{(q_1,q_2)^6}{q_1^6} \right)\right)\qquad \mathcal{I}_2=\mathcal{I}_3=<\!u_{3,0},u_{3,1},u_{3,2}\!> \quad \mathcal{I}_4=<\!u_{3,0},u_{3,1},u_{3,2},u_{3,3}\!>$\\
$q_1^{-1}\left(1+ \frac{5}{2}\frac{q_2^2}{q_1^2}+\frac{54}{2}\frac{q_3^2}{q_1^2}+u_{3,3}\frac{q_3^3}{q_1^3}+
\frac{175}{24}\frac{q_2^4}{q_1^4}+ \frac{15646365}{60604}\frac{q_2^2q_3^2}{q_1^4}+ 
\left(\frac{663499046129}{14245925733270}u_{3,3}^2+\frac{93293112339}{97889959} \right)\frac{q_3^4}{q_1^4}+
\frac{119177438935}{9191020828}\frac{q_2^2q_3^3}{q_1^5}+\right.$\\ $\left. \left(\frac{11164016229220693628221}{5008255522459502086348344}u_{3,3}^3+\frac{17921365151237750718}{239428834674960487}u_{3,3} \right)\frac{q_3^5}{q_1^5}+ O\left( \frac{(q_1,q_2)^6}{q_1^6} \right)\right)\qquad \mathcal{I}_2=\mathcal{I}_3=\mathcal{I}_4=<\!u_{3,0},u_{3,1},u_{3,2}\!>$\\
$q_1^{-1}\left(1+ \frac{5}{2}\frac{q_2^2}{q_1^2}+\frac{9}{2}\frac{q_3^2}{q_1^2}+u_{3,1}\frac{q_2^2q_3}{q_1^3}+u_{3,3}\frac{q_3^3}{q_1^3}+
\left(\frac{7}{324}u_{3,1}^2+\frac{175}{24}\right)\frac{q_2^4}{q_1^4}+ \left(\frac{19}{123}u_{3,1}^2+\frac{95}{246}u_{3,1}u_{3,3}+\frac{7965}{164}\right)\frac{q_2^2q_3^2}{q_1^4}+ 
\left(\frac{3289}{12060}u_{3,3}^2+\frac{13311}{536} \right)\frac{q_3^4}{q_1^4}+\left(\frac{3497470290125}{504750025968}u_{3,1}-\frac{664048660625}{504750025968}u_{3,3}\right) \frac{q_2^4q_3}{q_1^5}+\right. $\\ $\left. \left( \frac{2220239578205}{252375012984}u_{3,1}- \frac{921441864205}{252375012984}u_{3,3} \right)  \frac{q_2^2q_3^3}{q_1^5}+
\left(-\frac{187329008582632851}{386510093636424240}u_{3,1}+\frac{775252202045235367}{193270046818212120}u_{3,3}\right)\frac{q_3^5}{q_1^5}+ O\left( \frac{(q_1,q_2)^6}{q_1^6} \right)\right)$\\
$\mathcal{I}_2=\mathcal{I}_3=<\!u_{3,0},u_{3,2},\; \mathcal{I}_4=\left\langle\!u_{3,0},u_{3,2},u_{3,3}^5+\frac{189924461418873225}{1610583723485101}u_{3,3}^3-\frac{31218515091156436875}{6442334893940404}u_{3,3}, u_{3,3}^3+\frac{322885767353079225}{3221167446970202}u_{3,3}+\frac{103594358565795300}{1610583723485101}u_{3,1} \right\rangle$\\
$q_1^{-1}\left(1+ \frac{9}{2}\frac{q_2^2}{q_1^2}+\frac{27}{2}\frac{q_3^2}{q_1^2}+u_{3,0}\frac{q_2^3}{q_1^3}+u_{3,2}\frac{q_2q_3^2}{q_1^3}+
\left(\frac{3289}{12060}u_{3,0}^2+\frac{13311}{536}\right)\frac{q_2^4}{q_1^4}+ \left(\frac{42887}{2155950}u_{3,2}^2+\frac{47791}{143730}u_{3,0}u_{3,2}+\frac{376407}{1597}\right)\frac{q_2^2q_3^2}{q_1^4}+
\left(\frac{11446993}{324669060}u_{3,2}^2+\frac{376298379}{1603304} \right)\frac{q_3^4}{q_1^4}+\right.$\\
$\left. \left(\frac{207570596097655732674166455281418210645217993262286489869}{840193865795267874819893782623565474088703500063704000}u_{3,0}-
\frac{16787633864550203502743667842716175416793657678815071}{33607754631810714992795727304942618963548140002524160}u_{3,2}\right) \frac{q_2^5}{q_1^5}+
\left( \frac{4142079762228541965328315128069645141}{5646615134224665663290954120421260}u_{3,0}-\right.\right.$\\
$\left.\left. \frac{16869769031486182277210016047384863}{5646615134224665663290954120421260}u_{3,2} \right)  \frac{q_2^3q_3^2}{q_1^5}+
\left(\frac{2053379648674005395029029058639078197}{1389936033039917701733157937334464}u_{3,0}+\frac{121382613254803160782574481161900493}{6949680165199588508665789686672320}u_{3,2}\right)\frac{q_2q_3^4}{q_1^5}+ O\left( \frac{(q_1,q_2)^6}{q_1^6} \right)\right)$\\
$\mathcal{I}_2=\mathcal{I}_3=<\!u_{3,1},u_{3,3}\!>\quad \mathcal{I}_4=\left\langle\!u_{3,1},u_{3,3},
u_{3,0}^5-\frac{183692999458343787808210893960821881229602107524169}{20377836297273288008691439130795431294539450561232}u_{3,0}^3+.\frac{58718179609676651900557929966196287242373163213359375}{40755672594546596017382878261590862549078901122464}u_{3,0},\right.$\\ $\left.u_{3,0}^3-\frac{3554174443372635952956958960222053446124173871757419}{20377836297273288008691439130795431294539450561232}u_{3,0}+.\frac{135098452001995628889349592782686772373066739788025}{20377836297273288008691439130795431294539450561232}u_{3,2} \right\rangle$\\
$q_1^{-1}\left(1+ \frac{9}{2}\frac{q_2^2}{q_1^2}+\frac{35}{2}\frac{q_3^2}{q_1^2}+u_{3,0}\frac{q_2^3}{q_1^3}+u_{3,2}\frac{q_2q_3^2}{q_1^3}+u_{3,3}\frac{q_3^3}{q_1^3}+
\frac{3289}{12060}u_{3,0}^2\frac{q_2^4}{q_1^4}+ \left(\frac{299}{19730}u_{3,2}^2+\frac{6539}{19730}u_{3,0}u_{3,2}+\frac{1202985}{3946}\right)\frac{q_2^2q_3^2}{q_1^4}+ 
\frac{2171}{27090}u_{3,2}u_{3,3}\frac{q_2q_3^3}{q_1^4}+\left(\frac{1084733}{30333876}u_{3,2}^2+\frac{108970043}{1516693800}u_{3,3}^2+\right.\right.$\\
$\left.\left. \frac{8028474825}{20222584} \right)\frac{q_3^4}{q_1^4}+
\left(\frac{99171}{8576}u_{3,0}^3+\frac{363467}{4824000}u_{3,0}\right)\frac{q_2^5}{q_1^5}+ 
\left(\frac{12500754448143833461}{16013235835701805919400}u_{3,2}^3-\frac{106381137094520737}{9000423588396965634000}u_{3,3}^2u_{3,2}+ \frac{236392873}{2319952050} u_{3,0}^2u_{3,2}+\frac{157506321808073909150063}{17711147190979394229584}u_{3,2}+\right.\right.$\\
$\left.\left. \frac{203530243}{2667496}u_{3,0}\right)  \frac{q_2^3q_3^2}{q_1^5}+
\left(\frac{110817541821942436684566467}{1588606758916943651880510530000}u_{3,3}^3-\frac{385229250375178490966079}{12309652055185871405368}u_{3,3}  \right)\frac{q_2^2q_3^3}{q_1^5}+\left(\frac{2133042536605191925}{57080871628797312}u_{3,2}+\frac{6522159327691787}{1798047456307115328}u_{3,2}^3+\right.\right.$\\
$\left.\left. \frac{90416581343807626661}{12462124968549644724000}u_{3, 2}u_{3,3}^2  \right)\frac{q_2q_3^4}{q_1^5}+\left( \frac{179117095742904615655002021}{17581260547819220834716846}u_{3, 3}+\frac{45081551340734816448303196649189}{8334836094207397117218388767450000}u_{3, 3}^3 \right)\frac{q_3^5}{q_1^5}+ O\left( \frac{(q_1,q_2)^6}{q_1^6} \right)\right)$\\
$\mathcal{I}_2=\mathcal{I}_3=<\!u_{3,1}\!>\quad \mathcal{I}_4=\left\langle u_{3, 3}^3-\frac{269273638707190936992425}{6971035841270139539814}u_{3, 2}^2u_{3, 3}-\frac{19517430127889746138453959375}{60415643957674542678388}u_{3, 3}, \frac{2960366553926052591241012}{296036655392605259124101}u_{3, 0}u_{3, 3}^3-\frac{200153459825026311801301}{592073310785210518248202}u_{3, 2}u_{3, 3}^3-\right.$\\
$\left. \frac{956354076266597560784244009375}{296036655392605259124101}u_{3, 0}u_{3, 3}+\frac{163451037973980541410927748125}{296036655392605259124101}u_{3, 2}u_{3, 3}, -\frac{52777633186609423650515300}{46189259959621456569531}u_{3, 0}u_{3, 2}u_{3, 3}+u_{3, 3}^3-\frac{326902075947961082821855496250}{200153459825026311801301}u_{3, 3},\right.$\\
$\left. u_{3, 2}u_{3, 3}^2+\frac{769266973367262020}{815441353231641}u_{3, 0}u_{3, 2}^2-\frac{364191525409188890}{5164461903800393}u_{3, 2}^3+\frac{50935367483822301472125}{67138004749405109}u_{3, 2}\right\rangle$\\
$q_1^{-1}\left(1+ \frac{9}{2}\frac{q_2^2}{q_1^2}+\frac{54}{2}\frac{q_3^2}{q_1^2}+u_{3,0}\frac{q_2^3}{q_1^3}+u_{3,2}\frac{q_2q_3^2}{q_1^3}+u_{3,3}\frac{q_3^3}{q_1^3}+
\left(\frac{13311}{536}+\frac{3289}{12060}u_{3,0}^2\right)\frac{q_2^4}{q_1^4}+ \left(\frac{8131447}{843807240}u_{3,2}^2+\frac{1469946609}{3125212}+\frac{46560929}{140634540}u_{3, 0}u_{3,2}\right)\frac{q_2^2q_3^2}{q_1^4}+ 
\frac{26821}{525690}u_{3,2}u_{3,3}\frac{q_2q_3^3}{q_1^4}+\left(\frac{3843346549}{105721155720}u_{3, 2}^2+\right.\right.$\\
$\left.\left. \frac{93293112339}{97889959}+\frac{663499046129}{14245925733270}u_{3, 3}^2 \right)\frac{q_3^4}{q_1^4}+
\left(\frac{99171}{8576}u_{3, 0}+\frac{363467}{4824000}u_{3, 0}^3\right)\frac{q_2^5}{q_1^5}+ 
\left(\frac{118579147678287}{965637379396}u_{3, 0}+\frac{548369856819135481128554482500530673}{59393927254261565457430506022289096} u{3, 2}+\frac{21459422475171895871107095123965213}{70607855310476622219169250816049649200}u_{3, 2}^3+\right.\right.$\\
$\left.\left. \frac{1543078654120439}{14556983494394700}u_{3, 0}^2u_{3, 2}-\frac{65323338764275478512072478}{18007002218330647802772268712859}u_{3, 2}u_{3, 3}^2\right)  \frac{q_2^3q_3^2}{q_1^5}+\right.$\\
$\left. \left(-\frac{56399453236830854062047492240965325}{501489449847342564242033164136492}u_{3, 3}+\frac{1645058769784838360345278710391484}{4992954335042604405234742690433948475}u_{3, 3}^3  \right)\frac{q_2^2q_3^3}{q_1^5}+\left(\frac{6310841737996595281093353}{108326616514187004913604}u_{3, 2}+\frac{1027729317355996171389851}{438722796882457369900096200}u{3, 2}^3+\right.\right.$\\
$\left.\left. \frac{3015116344570439531602738498}{1009671449343745361987648367225}u_{3, 2}u_{3, 3}^2  \right)\frac{q_2q_3^4}{q_1^5}+\left( -\frac{4008495460000401399979252028072066166113121}{29916876551082911336452430097420793501516}u_{3, 3}+\frac{537412961592405139051640095845978898256260852721}{188802669023836001803818434124847208914422412210200}u_{3, 3}^3 \right)\frac{q_3^5}{q_1^5}+ O\left( \frac{(q_1,q_2)^6}{q_1^6} \right)\right)$\\
$\mathcal{I}_2=\mathcal{I}_3=<\!u_{3,1}\!>\quad \mathcal{I}_4=\left\langle u_{3, 3}^3-\frac{21405037493484133839598976518021}{4495112199619735814187970929088}u_{3, 2}^2u_{3, 3}-\frac{380181970518931304259885356994496125}{1123778049904933953546992732272}u_{3, 3},
\frac{4035092449451502815699543327165595075}{184299600184409168381706808092608}u_{3, 2}u_{3, 3}-\right.$\\
$\left. \frac{380181970518931304259885356994496125}{1123778049904933953546992732272}u_{3, 0}u_{3, 3}-\frac{51588237214057}{1157111566196588}u_{3, 2}u_{3, 3}^3+u{3, 0}u_{3, 3}^3,
-\frac{877606537232849487423558037238861}{8216745705879367033563172965712}u_{3, 0}u_{3, 2}u_{3, 3}+u_{3, 3}^3-\right.$\\
$\left.\frac{4035092449451502815699543327165595075}{8216745705879367033563172965712}u_{3, 3},
u_{3, 2}u_{3, 3}^2+\frac{58278294334415316709243773}{32748020504236080458960}u_{3, 0}u_{3, 2}^2-\frac{104659567580617123168215121}{1244424779160971057440480}u_{3, 2}^3+\frac{274668416871456936680905800555}{124442477916097105744048}u_{3, 2}\right\rangle$\\
\end{tiny}
\end{minipage}
\end{turn}
\end{center}

\begin{center}
\begin{turn}{-90}
\hspace{-0.5cm}
\begin{minipage}{22.5cm}
\begin{tiny}
$q_1^{-1}\left(1+ \frac{9}{2}\frac{q_2^2}{q_1^2}+\frac{44}{2}\frac{q_3^2}{q_1^2}+u_{3,0}\frac{q_2^3}{q_1^3}+u_{3,2}\frac{q_2q_3^2}{q_1^3}+
\left(\frac{3289}{12060}u_{3,0}^2+\frac{13311}{536}\right)\frac{q_2^4}{q_1^4}+ \left(\frac{48787}{4090760}u_{3,2}^2+\frac{2031899}{6136140}u_{3,0}u_{3,2}+\frac{78370281}{204538}\right)\frac{q_2^2q_3^2}{q_1^4}+
\left(\frac{4690489}{129913320}u_{3,2}^2+\frac{682538736}{1082611} \right)\frac{q_3^4}{q_1^4}+\right.$\\
$\left.\left(\frac{363467}{4824000}u_{3,0}^3+\frac{99171}{8576}u_{3,0}\right)\frac{q_2^5}{q_1^5}+
\left(\frac{442434627857883031737075841}{938912143429199537485438760400}u_{3,2}^3+\frac{8173166259659}{78016448675700}u_{3,0}^2u_{3,2}+ \frac{477961860375836325150938652440}{52422594674796974176270330789}u_{3,2}+\frac{85215635549}{862536746} u_{3,0}\right)  \frac{q_2^3q_3^2}{q_1^5}+\right.$\\
$\left.\left( \frac{28213191044570641817}{9796926716514400926600}u_{3,2}^3+\frac{1542788367059997027311}{32656422388381336422}u_{3,2} \right)\frac{q_2q_3^4}{q_1^5}+ O\left( \frac{(q_1,q_2)^6}{q_1^6} \right)\right)$\\
$\mathcal{I}_2=\mathcal{I}_3=<\!u_{3,1},u_{3,3}\!>\quad \mathcal{I}_4=\left\langle\!u_{3,1},u_{3,3},
u_{3,2}^3-\frac{1222814898902417148900}{70922785617530321}-\frac{3628491376486815158}{212768356852590963} u_{3,0}u_{3,2}^2\right\rangle$\\
$q_1^{-1}\left(1+ \frac{9}{2}\frac{q_2^2}{q_1^2}+\frac{20}{2}\frac{q_3^2}{q_1^2}+u_{3,0}\frac{q_2^3}{q_1^3}+u_{3,1}\frac{q_2^2q_3}{q_1^3}+u_{3,2}\frac{q_2q_3^2}{q_1^3}+u_{3,3}\frac{q_3^3}{q_1^3}+
\left(\frac{13311}{536}+\frac{3289}{12060}u_{3, 0}^2+\frac{11}{1809}u_{3, 1}^2\right)\frac{q_2^4}{q_1^4}+\left(\frac{44}{135}u_{3, 1}u_{3, 0}+\frac{11}{270}u_{3, 1}u_{3, 2}\right)\frac{q_2^3q_3}{q_1^4}+ \left(\frac{2959}{8772}u_{3, 0}u_{3, 2}+\frac{451}{2924}u_{3, 1}u_{3, 3}+\frac{11}{408}u_{3, 2}^2+\right.\right.$\\
$\left.\left. \frac{814}{6579}u_{3, 1}^2+\frac{258255}{1462}\right)\frac{q_2^2q_3^2}{q_1^4}+ 
\left( \frac{253}{2241}u_{3, 1}u_{3, 2}+\frac{121}{830}u_{3, 2}u_{3, 3}\right)\frac{q_2q_3^3}{q_1^4}+\left(\frac{2981}{86472}u_{3, 2}^2+\frac{153200}{1201}+\frac{105347}{840700}u_{3, 3}^2 \right)\frac{q_3^4}{q_1^4}+
\left(\frac{24121663415883}{456685207040}u_{3, 0}-\frac{36765303752462352960327967}{2196201151998323963013120}u_{3, 2}+\right.\right.$\\
$\left.\left. \frac{89405326300298824860496291}{17789229331186424100406272000}u_{3, 2}^3-\frac{41358147528355848356789}{822889690590546031104000}u_{3, 0}u_{3, 2}^2+\frac{1346853263812693}{13769058992256000}u_{3, 0}^2u_{3, 2}-\frac{922257609158044398539993}{45571173903965222232522240000}u_{3, 2}u_{3, 3}^2+\frac{363467}{4824000}u_{3, 0}^3+\frac{230463041841}{4566852070400}u_{3, 0}u_{3, 1}u_{3, 3}-\right.\right.$\\
$\left.\left. \frac{32105330815398760859953}{1517774318200340457369600}u_{3, 1}u_{3, 3}u_{3, 2}\right)\frac{q_2^5}{q_1^5}+ 
\left(\frac{1006362154607592365}{76272045092872776}u_{3, 1}+\frac{151684318285194825}{5403817053617488}u_{3, 3}-\frac{119025726972819079}{11265073504331898240}u_{3, 2}^2u_{3, 3}+\frac{23649118837130954071}{363298620514703718240}u_{3, 0}u_{3, 2}u_{3, 3}+\right.\right.$\\
$\left.\left.\frac{5798437746842933}{10723874942903904936}u_{3, 1}^2u_{3, 3}+\frac{3862054924584189703159}{85862492042850598854240}u_{3, 1}u_{3, 3}^2+\frac{2408236239104282863}{1367405224425023128128}u_{3, 1}u_{3, 2}^2+\frac{2556040771182469}{284638889347423632}u_{3, 1}u_{3, 0}u_{3, 2}\right)  \frac{q_2^4q_3}{q_1^5}+
\left(-\frac{7295524609850}{130123321413}u_{3, 1}+\frac{7533233096661}{105601052404}u_{3, 3}-\right.\right.$\\
$\left.\left. \frac{114781940299}{12230075371440}u_{3, 2}^2u_{3, 3}+\frac{204029796932797}{1314733102429800}u_{3, 0}u_{3, 2}u_{3, 3}-\frac{833363511404}{8316082876815}u_{3, 1}^2u_{3, 3}+\frac{2463793308929051}{26633641426812840}u_{3, 1}u_{3, 3}^2+\frac{46351753531771}{1166425453146132}u_{3, 1}u_{3, 2}^2-\frac{76656425825}{485605933866}u_{3, 1}u_{3, 0}u_{3, 2}\right)  \frac{q_2^3q_3^2}{q_1^5}+\right.$\\
$\left.\left(  -\frac{7295524609850}{130123321413}u_{3, 1}+\frac{7533233096661}{105601052404}u_{3, 3}-\frac{114781940299}{12230075371440}u_{3, 2}^2u_{3, 3}+\frac{204029796932797}{1314733102429800}u_{3, 0}u_{3, 2}u_{3, 3}-\frac{833363511404}{8316082876815}u_{3, 1}^2u_{3, 3}+\frac{2463793308929051}{26633641426812840}u_{3, 1}u_{3, 3}^2+\frac{46351753531771}{1166425453146132}u_{3, 1}u_{3, 2}^2-\right.\right.$\\
$\left.\left.\frac{76656425825}{485605933866}u_{3, 1}u_{3, 0}u_{3, 2} \right)\frac{q_2^2q_3^3}{q_1^5}+
\left(\frac{50525703076300}{161496415089}u_{3, 2}-\frac{4235472447613}{52324838488836}u_{3, 2}^3+\frac{3093278443}{4840858404}u_{3, 0}u_{3, 2}^2+\frac{4492719610495891}{187658834333418000}u_{3, 2}u_{3, 3}^2+\frac{13050298589}{34722700404}u_{3, 1}u_{3, 3}u_{3, 2} \right)\frac{q_2q_3^4}{q_1^5}+\right.$\\
$\left. \left( \frac{46809030}{1737847}u_{3, 3}+\frac{1362858557}{173089561200}u_{3, 2}^2u_{3, 3}+\frac{68250852673}{4257725150000}u_{3, 3}^3+\frac{133423}{51883632}u_{3, 1}u_{3, 2}^2 \right)\frac{q_3^5}{q_1^5}+ O\left( \frac{(q_1,q_2)^6}{q_1^6} \right)\right)$\\
$\mathcal{I}_2=\mathcal{I}_3=<\!0\!>\quad \mathcal{I}_4=\left\langle
\frac{21492255493447480629464816241}{96826350676115640555159100}u_{3,1}-\frac{161978738053543177347}{58408235543590696}u_{3,0}+\frac{3682769787616585454443049170788929}{4017640277688287994053545828000}u_{3,2}-\frac{153908497065685013278741413}{506763343612357343217200}u_{3,3}+\frac{83773728539}{2056568535750}u_{3,1}^3+\right.$\\
$\left.\frac{3602135236781333}{117779680042402500}u_{3,1}^2u_{3,2}-\frac{208416169135127627446573912789223}{756811308122677505856598167600000}u_{3,2}^3+\frac{7730882033}{41131370715}u_{3,0}u_{3,1}^2+\frac{4397878855122589301767627048843}{1505360636935661613092502600000}u_{3,0}u_{3,2}^2+\frac{111825024093141519961813257}{978171105112224639233200000}u_{3,2}^2u_{3,3}-\right.$\\
$\left.\frac{19202818906143244099}{2920411777179534800}u_{3,0}^2u_{3,2}+\frac{431284635777701402102430913273}{570219122859315840469296005000000}u_{3,2}u_{3,3}^2-\frac{7406139134798571607952467731}{10515339379956414871756900000}u_{3,0}u_{3,2}u_{3,3}-\frac{1981359607679726589}{584082355435906960}u_{3,0}u_{3,1}u_{3,3}+\right.$\\
$\left.\frac{16159160375636350412725999632347}{13882770318406657098519746200000}u_{3,1}u_{3,3}u_{3,2}+\frac{11541178916148758193157317}{36516770348537514437710000}u_{3,1}^2u_{3,3}-\frac{3628412075057407138506005535897}{7455629002060904322747250700000}u_{3,1}u_{3,3}^2-\frac{16997850549154609014088319603}{224276849473049251146368520000}u_{3,1}u_{3,2}^2+\right.$\\
$\left.\frac{209831610644740256848667017}{334578980271340473192265000}u_{3,1}u_{3,0}u_{3,2}+u_{3,1}u_{3,0}^2,
\frac{11023564040554427666539761429}{61968864432714009955301824}u_{3,1}-\frac{809893690267715886735}{467265884348725568}u_{3,0}+\frac{9500305040852292249125132917333559}{16456254577411227623643323711488}u_{3,2}-\right.$\\
$\left.\frac{91088276984695530979304697}{324328539911908699659008}u_{3,3}+\frac{12565459497}{438734620960}u_{3,1}^3+\frac{18240141713679571}{964851138907361280}u_{3,1}^2u_{3,2}-\frac{537628219940107140701558350542803}{3099899118070487063988626094489600}u_{3,2}^3+\frac{7730882033}{65810193144}u_{3,0}u_{3,1}^2+\frac{11317999279608608041128669156553}{6165957168888469967226890649600}u_{3,0}u_{3,2}^2+\right.$\\
$\left.\frac{66181848063068086388215533}{626029507271823769109248000}u_{3,2}^2u_{3,3}-\frac{19202818906143244099}{4672658843487255680}u_{3,0}^2u_{3,2}+\frac{1133842119622475310043992533923}{2335617527231757682562236436480000}u_{3,2}u_{3,3}^2-\frac{4383204733717948930008199839}{6729817203172105517924416000}u_{3,0}u_{3,2}u_{3,3}-\right.$\\
$\left.\frac{1981359607679726589}{934531768697451136}u_{3,0}u_{3,1}u_{3,3}+\frac{41676784231893787208363119215647}{56863827224193667475536880435200}u_{3,1}u_{3,3}u_{3,2}+\frac{100452525304217534839407441}{397302461392088157082284800}u_{3,1}^2u_{3,3}-\frac{2147417526703479037810755093693}{4771602561318978766558240448000}u_{3,1}u_{3,3}^2-\right.$\\
$\left.\frac{8840853261257999339095182607}{143537183662751520733675852800}u_{3,1}u_{3,2}^2+\frac{327022808578066208908275469}{642391642120973708529148800}u_{3,1}u_{3,0}u_{3,2}+u_{3,1}u_{3,0}^2, \frac{8119705174366086529811522577}{48413175338057820277579550}u_{3,1}-\frac{161978738053543177347}{116816471087181392}u_{3,0}+\right.$\\
$\left.\frac{14870337582300636982677825606451679}{32141122221506303952428366624000}u_{3,2}-\frac{69826208977705124409476661}{253381671806178671608600}u_{3,3}+\frac{8860106111}{342761422625}u_{3,1}^3+\frac{28427099781479791}{1884474880678440000}u_{3,1}^2u_{3,2}-\frac{841516266837011306783082446930723}{6054490464981420046852785340800000}u_{3,2}^3+\frac{7730882033}{82262741430}u_{3,0}u_{3,1}^2+\right.$\\
$\left.\frac{17705802748364933031308092640593}{12042885095485292904740020800000}u_{3,0}u_{3,2}^2+\frac{50733504972972215377690329}{489085552556112319616600000}u_{3,2}^2u_{3,3}-\frac{19202818906143244099}{5840823554359069600}u_{3,0}^2u_{3,2}+\frac{1782356720094574301990049414523}{4561752982874526723754368040000000}u_{3,2}u_{3,3}^2-\right.$\\
$\left.\frac{3360065420713577247167215107}{5257669689978207435878450000}u_{3,0}u_{3,2}u_{3,3}-\frac{1981359607679726589}{1168164710871813920}u_{3,0}u_{3,1}u_{3,3}+\frac{65231513945362190404335844984847}{111062162547253256788157969600000}u_{3,1}u_{3,3}u_{3,2}+\frac{73950500303148142428237033}{310392547962568872720535000}u_{3,1}^2u_{3,3}-\right.$\\
$\left. \frac{1646161613169798251144263874409}{3727814501030452161373625350000}u_{3,1}u_{3,3}^2-\frac{6539529995537929093658254091}{112138424736524625573184260000}u_{3,1}u_{3,2}^2+\frac{120905494589403709828247761}{250934235203505354894198750}u_{3,1}u_{3,0}u_{3,2}+u_{3,1}u_{3,0}^2, \frac{73712383302787538188228465761}{387305402704462562220636400}u_{3,1}-\right.$\\
$\left.\frac{485936214160629532041}{233632942174362784}u_{3,0}+\frac{177747758828436167967605762660940333}{257128977772050431619426932992000}u_{3,2}-\frac{582369803672098936042126773}{2027053374449429372868800}u_{3,3}+\frac{790993099357}{24678822429000}u_{3,1}^3+\frac{38119456437956873}{1675088782825280000}u_{3,1}^2u_{3,2}-\right.$\\
$\left.\frac{1117659694506847966268117264922169}{5381769302205706708313586969600000}u_{3,2}^3+\frac{7730882033}{54841827620}u_{3,0}u_{3,1}^2+\frac{23544244037889590806056576606779}{10704786751542482581991129600000}u_{3,0}u_{3,2}^2+\frac{423131396695772871698910297}{3912684420448898556932800000}u_{3,2}^2u_{3,3}-\frac{57608456718429732297}{11681647108718139200}u_{3,0}^2u_{3,2}+\right.$\\
$\left.\frac{21102189756216048155978964167121}{36494023862996213790034944320000000}u_{3,2}u_{3,3}^2-\frac{28023870521327646500034450051}{42061357519825659487027600000}u_{3,0}u_{3,2}u_{3,3}-\frac{5944078823039179767}{2336329421743627840}u_{3,0}u_{3,1}u_{3,3}+\frac{779803805130262341084870907540869}{888497300378026054305263756800000}u_{3,1}u_{3,3}u_{3,2}+\right.$\\
$\left.\frac{672102404038906146804563469}{2483140383700550981764280000}u_{3,1}^2u_{3,3}-\frac{13729440986554737069516573973737}{29822516008243617290989002800000}u_{3,1}u_{3,3}^2-\frac{176542672683304052381657155289}{2691322193676591013756422240000}u_{3,1}u_{3,2}^2+\frac{6532828008174407727920942563}{12044843289768257034921540000}u_{3,1}u_{3,0}u_{3,2}+u_{3,1}u_{3,0}^2\right\rangle$\\
\end{tiny}
\end{minipage}
\end{turn}
\end{center}

\label{}
\bibliographystyle{model1-num-names}
\bibliography{bibthese}

\begin{thebibliography}{19}
\expandafter\ifx\csname natexlab\endcsname\relax\def\natexlab#1{#1}\fi
\providecommand{\bibinfo}[2]{#2}
\ifx\xfnm\relax \def\xfnm[#1]{\unskip,\space#1}\fi
\bibitem[{Bruns(1887)}]{46}
\bibinfo{author}{H.~Bruns},
\newblock \bibinfo{journal}{Acta Math.} \bibinfo{volume}{11}
  (\bibinfo{year}{1887}).
\bibitem[{Poincar{\'e}(1890)}]{46b}
\bibinfo{author}{H.~Poincar{\'e}},
\newblock \bibinfo{title}{Sur le probl{\`e}me des trois corps et les
  {\'e}quations de la dynamique},
\newblock \bibinfo{journal}{Acta mathematica} \bibinfo{volume}{13}
  (\bibinfo{year}{1890}) \bibinfo{pages}{A3--A270}.
\bibitem[{Julliard~Tosel(1999)}]{47}
\bibinfo{author}{E.~Julliard~Tosel},
\newblock \bibinfo{title}{Non-int{\'e}grabilit{\'e} alg{\'e}brique et
  m{\'e}romorphe de probl{\`e}mes de n corps}  (\bibinfo{year}{1999}).
\bibitem[{Morales-Ruiz and Simon(2009)}]{5}
\bibinfo{author}{J.~Morales-Ruiz}, \bibinfo{author}{S.~Simon},
\newblock \bibinfo{title}{On the meromorphic non-integrability of some
  {$N$}-body problems},
\newblock \bibinfo{journal}{Discrete and Continuous Dynamical Systems (DCDS-A)}
  \bibinfo{volume}{24} (\bibinfo{year}{2009}) \bibinfo{pages}{1225--1273}.
\bibitem[{Boucher(2000)}]{68}
\bibinfo{author}{D.~Boucher},
\newblock \bibinfo{title}{Sur la non-int{\'e}grabilit{\'e} du probleme plan des
  trois corps de masses {\'e}galesa un le long de la solution de lagrange},
\newblock \bibinfo{journal}{CR Acad. Sci. Paris} \bibinfo{volume}{331}
  (\bibinfo{year}{2000}) \bibinfo{pages}{391--394}.
\bibitem[{Tsygvintsev(2007)}]{69}
\bibinfo{author}{A.~Tsygvintsev},
\newblock \bibinfo{title}{On some exceptional cases in the integrability of the
  three-body problem},
\newblock \bibinfo{journal}{Celestial Mechanics and Dynamical Astronomy}
  \bibinfo{volume}{99} (\bibinfo{year}{2007}) \bibinfo{pages}{23--29}.
\bibitem[{Nomikos and Papageorgiou(2009)}]{115}
\bibinfo{author}{D.~Nomikos}, \bibinfo{author}{V.~Papageorgiou},
\newblock \bibinfo{title}{Non-integrability of the anisotropic stormer problem
  and the isosceles three-body problem},
\newblock \bibinfo{journal}{Physica D: Nonlinear Phenomena}
  \bibinfo{volume}{238} (\bibinfo{year}{2009}) \bibinfo{pages}{273--289}.
\bibitem[{Albouy and Kaloshin(2012)}]{93}
\bibinfo{author}{A.~Albouy}, \bibinfo{author}{V.~Kaloshin},
\newblock \bibinfo{title}{Finiteness of central configurations of five bodies
  in the plane},
\newblock \bibinfo{journal}{Annals of mathematics} \bibinfo{volume}{176}
  (\bibinfo{year}{2012}) \bibinfo{pages}{535--588}.
\bibitem[{Morales-Ruiz et~al.(2007)Morales-Ruiz, Ramis, and Sim\'o}]{2}
\bibinfo{author}{J.~Morales-Ruiz}, \bibinfo{author}{J.~Ramis},
  \bibinfo{author}{C.~Sim\'o},
\newblock \bibinfo{title}{Integrability of {H}amiltonian systems and
  differential {G}alois groups of higher variational equations},
\newblock \bibinfo{journal}{Annales scientifiques de l'Ecole normale
  sup{\'e}rieure} \bibinfo{volume}{40} (\bibinfo{year}{2007})
  \bibinfo{pages}{845--884}.
\bibitem[{Combot(2013)}]{91}
\bibinfo{author}{T.~Combot},
\newblock \bibinfo{title}{A note on algebraic potentials and {M}orales-{R}amis
  theory},
\newblock \bibinfo{journal}{Celestial Mechanics and Dynamical Astronomy}
  (\bibinfo{year}{2013}) \bibinfo{pages}{1--22}.
  \bibinfo{note}{10.1007/s10569-013-9470-2}.
\bibitem[{Din(2003)}]{74}
\bibinfo{author}{M.~S.~E. Din},
\newblock \bibinfo{title}{Raglib: A library for real algebraic geometry},
\newblock \bibinfo{journal}{http://www-calfor.lip6.fr/~safey/RAGLib/}
  (\bibinfo{year}{2003}).
\bibitem[{A.~Bostan(2014)}]{116}
\bibinfo{author}{M.~S.-E.-D. A.~Bostan, T.~Combot},
\newblock \bibinfo{title}{Computing necessary integrability conditions for
  planar parametrized homogeneous potentials},
\newblock \bibinfo{journal}{ISSAC 2014}  (\bibinfo{year}{2014}).
\bibitem[{Moulton(1910)}]{42}
\bibinfo{author}{F.~Moulton},
\newblock \bibinfo{title}{The straight line solutions of the problem of n
  bodies},
\newblock \bibinfo{journal}{The Annals of Mathematics} \bibinfo{volume}{12}
  (\bibinfo{year}{1910}) \bibinfo{pages}{1--17}.
\bibitem[{Roberts(1999)}]{94b}
\bibinfo{author}{G.~E. Roberts},
\newblock \bibinfo{title}{A continuum of relative equilibria in the five-body
  problem},
\newblock \bibinfo{journal}{Physica D: Nonlinear Phenomena}
  \bibinfo{volume}{127} (\bibinfo{year}{1999}) \bibinfo{pages}{141--145}.
\bibitem[{Hampton and Moeckel(2006)}]{43}
\bibinfo{author}{M.~Hampton}, \bibinfo{author}{R.~Moeckel},
\newblock \bibinfo{title}{Finiteness of relative equilibria of the four-body
  problem},
\newblock \bibinfo{journal}{Inventiones Mathematicae} \bibinfo{volume}{163}
  (\bibinfo{year}{2006}) \bibinfo{pages}{289--312}.
\bibitem[{Combot(2011)}]{10}
\bibinfo{author}{T.~Combot},
\newblock \bibinfo{title}{Integrable homogeneous potentials of degree -1 in the
  plane with small eigenvalues},
\newblock \bibinfo{journal}{arXiv:1110.6130}  (\bibinfo{year}{2011}).
\bibitem[{Combot(2013)}]{9}
\bibinfo{author}{T.~Combot},
\newblock \bibinfo{title}{Integrability conditions at order $2$ for homogeneous
  potentials of degree $-1$},
\newblock \bibinfo{journal}{Non-linearity} \bibinfo{volume}{26}
  (\bibinfo{year}{2013}).
\bibitem[{Combot and Koutschan(2012)}]{90}
\bibinfo{author}{T.~Combot}, \bibinfo{author}{C.~Koutschan},
\newblock \bibinfo{title}{Third order integrability conditions for homogeneous
  potentials of degree $-1$},
\newblock \bibinfo{journal}{Journal of Mathematical Physics}
  \bibinfo{volume}{53} (\bibinfo{year}{2012}).
\bibitem[{Pacella(1987)}]{22}
\bibinfo{author}{F.~Pacella},
\newblock \bibinfo{title}{Central configurations of the n-body problem via
  equivariant morse theory},
\newblock \bibinfo{journal}{Archive for Rational Mechanics and Analysis}
  \bibinfo{volume}{97} (\bibinfo{year}{1987}) \bibinfo{pages}{59--74}.

\end{thebibliography}

\end{document}